\documentclass{amsart}

\usepackage{tikz}
\usepackage{tikz-cd}

\usetikzlibrary{shapes.misc}
\usepackage{amsmath,amscd,dsfont}
\usepackage{amssymb}
\usepackage{amsthm}
\usepackage{oldgerm}
\usepackage[pdftex,hidelinks]{hyperref}
\usepackage{multicol}
\newcommand{\ev}{{\sf ev}}
\newcommand{\Sh}{{\sf Sh}}

\renewcommand{\d}{{\rm{d}}}
\newcommand{\ra}{{\rightarrow}}
\newcommand{\Z}{\mathbb{Z}}

\newcommand{\sL}{\mathcal{L}}

\newcommand{\Hom}{{\sf Hom}}

\newcommand{\sym}{{\sf sym}}

\newcommand{\id}{{\sf id}}

\newcommand{\bone}{{\mathds{1}}}

\newcommand{\imag}{{\sf Im}}

\newtheorem{thm}{Theorem}[section]
\newtheorem{lem}[thm]{Lemma}

\newtheorem{conj}[thm]{Conjecture}
\newtheorem{cor}[thm]{Corollary}
\newtheorem{prop}[thm]{Proposition}
\newtheorem{rmk}[thm]{Remark}

\newtheorem{defi}[thm]{Definition}

\makeatletter
\@namedef{subjclassname@2020}{%
	\textup{2020} Mathematics Subject Classification}
\makeatother

\title{The inverse function theorem for curved L-infinity spaces }
\author[Amorim]{Lino Amorim}
\address{Department of Mathematics\\ Kansas State University\\ 138 Cardwell Hall, 1228 N. 17th Street\\
	Manhattan, KS 66506\\ USA}
\email{lamorim@ksu.edu}

\author[Tu]{Junwu Tu}
\address{Institute of Mathematical Sciences\\ ShanghaiTech University\\ 393 Middle Huaxia Road, Pudong New District, Shanghai, China, 201210.}
\email{tujw.at.shanghaitech.edu.cn}

\begin{document}

\begin{abstract}
	In this paper we prove an inverse function theorem in derived differential geometry. More concretely, we show that a morphism of curved $L_\infty$ spaces which is a quasi-isomorphism at a point has a local homotopy inverse. This theorem simultaneously generalizes the inverse function theorem for smooth manifolds and the Whitehead theorem for $L_\infty$ algebras. The main ingredients are the obstruction theory for $L_\infty$ homomorphisms (in the curved setting) and the homotopy transfer theorem for curved $L_\infty$ algebras. Both techniques work in the $A_\infty$ case as well. 
\end{abstract}
	
\keywords{$L_\infty$ space, inverse function theorem, Derived differential geometry}

\subjclass[2020]{14A30, 18N99, 58H15}

	\maketitle

\section{Introduction}

\subsection{Curved $L_\infty$ spaces} The notion of \emph{curved $L_\infty$ space} was introduced by Costello~\cite{Cos}, as an alternative approach to derived differential geometry. In {\sl loc. cit.} an $L_\infty$ space is defined as a pair $(M,\mathfrak{G})$ where $M$ is a smooth manifold, and $\mathfrak{G}$ is a \emph{curved $L_\infty$ algebra}\footnote{See \cite{marklB} for a definition.} over the de Rham algebra $\Omega_M^*$. In this paper, we shall work with a more down-to-earth notion of $L_\infty$ space, following analogous constructions in the theory of dg-schemes by Behrend~\cite{Beh1}\cite{Beh2},  and Ciocan-Fontanine--Kapranov~\cite{CFK1}\cite{CFK2}. This notion is however equivalent to the original one -  see~\cite{Tu} for a proof of the equivalence. Interestingly, a related concept also appeared in the study of deformations of coisotropic submanifolds in symplectic geometry\cite{OP,LO}. Conceptually, the $L_\infty$ approach to derived geometric structures is Koszul dual to the more classical approach using dg (or simplicial) commutative algebras as developed by To\"en--Vezzosi~\cite{TV}. Such structures naturally appear in various gauge theories, producing $L_\infty$-enhancements of the associated Maurer--Cartan moduli spaces. 

More precisely, throughout the paper, an $L_\infty$ space $\mathbb{M}=(M, \mathfrak{g})$ is given by a pair  of $M$ a smooth manifold and $\mathfrak{g}$ a curved $L_\infty$ algebra over the ring of smooth functions $C^\infty(M)$. We also require that $\mathfrak{g}$ is of the form
\[ \mathfrak{g}= \mathfrak{g}_2\oplus \mathfrak{g}_3\oplus\cdots\oplus \mathfrak{g}_d,\]
where each $\mathfrak{g}_i$ is a vector bundle in degree $i$. In particular it has minimal degree  $2$ and maximal degree $d$ for some $d\geq 2$. Conceptually, this grading condition reflects the fact that we are interested in derived schemes, not stacks. 

The main goal of this paper is to understand when maps of $L_\infty$ spaces are ``invertible". Let us introduce some terminology and notation in order to describe our results. Throughout the paper, we use the notation
\[ \mu_k: \mathfrak{g}[1] \otimes \cdots\otimes \mathfrak{g}[1] \to \mathfrak{g}[1], \;\; k \geq 0\]
to denote the (shifted) higher brackets of an $L_\infty$ algebra. We shall use the same notation in the $A_\infty$ case as well. Note that in the $L_\infty$ case, the maps $\mu_k$ are all graded symmetric by definition. Given an $L_\infty$ space $\mathbb{M}=(M, \mathfrak{g})$, the curvature term $\mu_0\in \mathfrak{g}_2$ is a section of the bundle $\mathfrak{g}_2$. Let $p\in \mu_0^{-1}(0)$ be a point in its zero locus. The tangent complex at $p$ is defined as the following chain complex
\[ T_p \mathbb{M}:= T_pM \stackrel{\d\mu_0|_p}{\longrightarrow}  \mathfrak{g}_2|_p \stackrel{\mu_1|_p}{\longrightarrow} \mathfrak{g}_3|_p\stackrel{\mu_1|_p}{\longrightarrow}\cdots\]
The fact that this is a chain complex follows from the $L_\infty$ algebra equation together with $\mu_0(p)=0$. A morphism between two curved $L_\infty$ spaces $\mathbb{M}=(M,\mathfrak{g})$ and $\mathbb{N}=(N,\mathfrak{h})$ is given by a pair $\mathfrak{f}=(f,f^\sharp)$ where $f: M\ra N$ is a smooth map and $f^\sharp=(f^\sharp_1,f^\sharp_2,\ldots): \mathfrak{g}\ra f^*\mathfrak{h}$ is a sequence of bundle maps which define a $L_\infty$ homomorphism.

Let $\mathfrak{f}: \mathbb{M}\ra \mathbb{N}$ be a morphism and $p\in M$ a point in the zero locus of the curvature of $\mathfrak{g}$. This morphism induces a map between the tangent complexes $d\mathfrak{f}_p: T_p\mathbb{M} \ra T_{f(p)} \mathbb{N}$, explicitly given by
\[ \begin{CD}
 T_pM @>>> \mathfrak{g}_2|_p @>>>  \mathfrak{g}_3|_p @>>> \cdots \\
 @V df_p VV    @V f_1^\sharp VV      @V f_1^\sharp VV     @. \\
T_{f(p)} N @>>> \mathfrak{h}_2|_{f(p)} @>>> \mathfrak{h}_3|_{f(p)} @>>>  \cdots
\end{CD}\]
We are now ready to state our first main result which states that if $d\mathfrak{f}_p$ is a quasi-isomorphism of chain complexes then $\mathfrak{f}$ is ``locally invertible".

\begin{thm}\label{thm:main}
	Let $(M,\mathfrak{g})$ and $(N,\mathfrak{h})$ be  $L_\infty$ spaces and $\mathfrak{f}=(f, f^\sharp): (M,\mathfrak{g}) \to (N,\mathfrak{h})$ be a $L_\infty $ morphism. Assume that the tangent map $d\mathfrak{f}_p$ is a quasi-isomorphism at $p\in (\mu_0^\mathbb{M})^{-1}(0)$. Then there exist open neighborhoods $U$ of $p$ and $V$ of $f(p)$ such that the restriction
\[ \mathfrak{f}|_U: (U, \mathfrak{g}|_U) \ra (V, \mathfrak{h}|_V)\]
is a homotopy equivalence\footnote{The notion of homotopy between $L_\infty$ spaces is formulated in Definition~\ref{def:homotopy}.} of $L_\infty$ spaces.
\end{thm}

In the case when both $\mathfrak{g}$ and $\mathfrak{h}$ are trivial, this is simply the inverse function theorem for smooth manifolds. 
When the $L_\infty$ bundle $\mathfrak{g}$ is concentrated in degree $2$, we obtain the notion of $m$-Kuranishi neighborhood in the work of Joyce \cite{Joy}, or Kuranishi chart (with trivial isotropy) introduced by Fukaya--Oh--Ohta--Ono~\cite{FOOO}. In this special case, our theorem essentially recovers Theorem 4.16 in \cite{Joy}, but with a slightly different notion of homotopy.
In the case when both $M$ and $N$ are a point, the above theorem recovers the well known fact (see~\cite{Pro} and~\cite{LodVal}) that quasi-isomorphisms between (uncurved) $L_\infty$ algebras are homotopy equivalences - this statement is referred to as the Whitehead theorem for  $L_\infty$ algebras in \cite{FOOO}. 

In forthcoming work, we shall use the complex analytic version of the above theorem to prove local invariance of the Maurer--Cartan moduli space associated with {\em bounded} $L_\infty$ algebras (see~\cite{Tu}). This result is essential to understand $L_\infty$ enhancements of moduli spaces from gauge theory, such as moduli spaces of flat connections on a vector bundle.

We also expect a global version of Theorem~\ref{thm:main} to hold. We formulate it in the following conjecture. 

\begin{conj}
Let $\mathfrak{f}=(f, f^\sharp): (M,\mathfrak{g}) \to (N,\mathfrak{h})$ be a $L_\infty $ morphism. Assume that the tangent map $d\mathfrak{f}_p$ is a quasi-isomorphism for all $p\in (\mu_0^\mathbb{M})^{-1}(0)$, and that the induced map $f: (\mu_0^\mathbb{M})^{-1}(0) \to (\mu_0^\mathbb{N})^{-1}(0)$ on the zero loci is a bijection. Then there exist open neighborhoods $U$ of $(\mu_0^\mathbb{M})^{-1}(0)$ and $V$ of $(\mu_0^\mathbb{N})^{-1}(0)$ such that the restriction
\[ \mathfrak{f}|_U: (U, \mathfrak{g}|_U) \ra (V, \mathfrak{h}|_V)\]
is a homotopy equivalence\footnote{See Definition \ref{def:homotopy} for homotopy equivalence of $L_\infty$ spaces. In this global setting it requires the existence of torsion-free, flat connections on $M$ and $N$. } of $L_\infty$ spaces.
\end{conj}

At this moment we are able to prove this conjecture on the special case where both $\mathfrak{g}$ and $\mathfrak{h}$ are concentrated in degree $2$. Our proof of this result uses a  partition-of-unity argument similar to the one used by Joyce in \cite{Joy} where an analogous result is proved for $m$-Kuranishi spaces. 

The second main result of the paper is the construction of a {\em minimal chart} around a point $p\in \mu_0^{-1}(0)\subset M$. This construction generalizes the so-called minimal model construction for (uncurved) $L_\infty$ algebras (see \cite{DHR, V, LodVal}). More precisely, we have the following

\begin{thm}\label{thm:main2}
Let $\mathbb{M}=(M,\mathfrak{g})$ be a $L_\infty$ space and $p\in \mu_0^{-1}(0)$. Then there exists a $L_\infty$ space $\mathbb{W}=(W,\mathfrak{h})$ with $W\subset M$ a submanifold containing $p$, such that
\begin{itemize}
\item The tangent complex $T_p\mathbb{W}$ has zero differential.
\item There exists an open neighborhood $U\subset M$ of $p$  which contains $W$ and such that the inclusion map $i: W \to U$ extends to a homotopy equivalence
\[ (i, i^\sharp): (W,\mathfrak{h}) \ra (U, \mathfrak{g}|_U)\]
of $L_\infty$ spaces. Here the map $i^\sharp: \mathfrak{h}\ra i^*\mathfrak{g}$ is an $L_\infty$ homomorphism constructed explicitly using summation over trees. We refer to Subsection~\ref{subsec:minimalchart} for details.
\end{itemize}
\end{thm}

This theorem is one of the main ingredients in our proof of Theorem \ref{thm:main} but we expect it will have many other applications in derived differential geometry. For instance, the analogous statement in Derived Algebraic Geometry is an important step in the proof of the Darboux theorem for shifted symplectic derived schemes in \cite{BBJ}.

\subsection{About the proofs} The proof of Theorem~\ref{thm:main2} generalizes the homotopy transfer theorem (sometimes also called homological perturbation lemma) to the curved setting. We would like to point out that we do not impose any conditions on the curvature term, unlike the filtered case considered in \cite{F, FOOO}. In that case, one assumes the algebra is equipped with a filtration and the curvature term lives in the positive part of the filtration. For these filtered algebras the transfer theorem was stated in \cite{F} and proved in detail for $A_\infty$ algebras in \cite{FOOO}. Our theorem is valid for general curved $A_\infty$ or $L_\infty$ algebras in the presence of a generalization of the usual homotopy retraction data $(i, p, H)$. We refer to Section~\ref{sec:transfer} for more details, but for example, the homotopy operator $H$ must satisfy
\[ \mu_{1}H+H \mu_{1} = ip -\id - H \mu_1^2 H,\]
which reduces to the usual homotopy identity in the uncurved case since $\mu_0=0$ implies that $\mu_1^2=0$. 
Remarkably, both the minimal $L_\infty$ structure on $\mathfrak{h}$ and $L_\infty$ homomorphism $i^\sharp$ are given by summation over the same stable trees as in the uncurved setting, and the curvature term does not play a big role (see Section~\ref{sec:transfer}). In that sense the construction is closer to the uncurved case than to the filtered one studied in \cite{F, FOOO}.

For the proof of Theorem~\ref{thm:main} we first develop an obstruction theory for $L_\infty$  (or $A_\infty$ ) homomorphisms between curved $L_\infty$ (respectively $A_\infty$ ) algebras, again without making any additional assumptions on the ground ring or on the curvature term. Since we are in the general curved setting (and therefore have no good notion of convergence), the homomorphisms we consider are by definition strict (in the terminology of \cite{FOOO}), meaning the constant terms $f_0$ vanish. Therefore our theory is once again closer to the uncurved case than the filtered case for which an analogous theory was developed in~\cite{FOOO}.

Our obstruction theory is set up in a way which easily generalizes from $L_\infty$ algebras to spaces. Combining this new obstruction theory with Theorem~\ref{thm:main2}  we  prove Theorem~\ref{thm:main} using an argument analogous to that in~\cite{FOOO}. 

We present our results in the smooth realm, meaning all the manifolds (and vector bundles, maps,...) are real $C^\infty$ manifolds. But our proofs and constructions also work in the complex analytic setting, with only minor modifications, see Remark \ref{rmk:analytic_homotopy}.

\subsection{Other related works} A recent preprint~\cite{BLX} by Behrend--Liao--Xu obtained  similar results, using the framework of categories of fibrant objects. While the definitions of $L_\infty$ spaces~\footnote{In {\sl Loc. Cit.} the authors use derived manifolds.} and the tangent complexes are clearly the same as ours, it is not clear at the moment of writing how Behrend--Liao--Xu's notion of homotopy between morphisms of $L_\infty$ spaces is related to ours (see Definition~\ref{def:homotopy}). One could say our approach is more algebraic in the sense that both the obstruction theory and the homological perturbation technique are generalizations of the situation for uncurved algebras.

An interesting question is whether Theorem~\ref{thm:main} and Theorem~\ref{thm:main2} admit generalizations to allow the tangent complex to have components in non-positive degrees. This is related to the notion of  {\em shifted Lie algebroid structure} studied by Pym--Safronov~\cite{PS}.

\subsection{Organization of the paper} In Section~\ref{sec:obstruction}, we develop the obstruction theory for constructing $A_\infty$/$L_\infty$ homomorphisms between curved $A_\infty$/$L_\infty$ algebras. Section~\ref{sec:transfer} generalizes the homotopy transfer theorem to the curved setting. In Section~\ref{sec:spaces} we recall basic definitions of $L_\infty$ spaces. In particular, we explicitly describe homotopies between $L_\infty$ morphisms. In Section~\ref{sec:theorems} we prove Theorem~\ref{thm:main} and Theorem ~\ref{thm:main2}.

\subsection{Acknowledgments.} We are grateful to Jim Stasheff for sending us his comments and suggestions of an earlier version of the paper.

\section{Obstruction theory}\label{sec:obstruction}

Let $(A,\mu_0^A,\mu_1^A,\cdots)$ and $(B,\mu_0^B,\mu_1^B,\cdots)$ be two curved $A_\infty$/$L_\infty$ algebras over a commutative ring $R$. In this section, we study the obstruction theory of $A_\infty$/$L_\infty$ homomorphisms from $A$ to $B$. Classically, in the non-curved case, this is done by using a pro-nilpotent $L_\infty$ algebra structure on the space $\Hom( T^c A[1], B)$ of cochains on $A$ with values in $B$. In particular, the obstruction class to construct the $(n+1)$-th component of a $A_\infty$/$L_\infty$ homomorphism $(f_j)_{j=1}^n$ from  $A$ to $B$ is a cohomology class
\[ \mathfrak{o}( (f_j)_{j=1}^n) \in H^1\big( \Hom( A[1]^{n+1}, B[1]),d=[\mu_1,-] \big),\]   
determined by the first $n$ components. However, adding the curvature term spoils the pro-nilpotent structure, and obviously the above obstruction space is not even defined as $\mu_1^A$ and $\mu_1^B$ might not square to zero. In this section, we define a variant of the above obstruction space which takes into account the appearance of curvatures, which allows us to extend the obstruction theory of $A_\infty$/$L_\infty$ homomorphisms to the curved setting.

\subsection{Definition of obstruction spaces.}\label{subsec:obs}
First, using the curvature term $\mu_0^A$, we form the complex $C(A,B)$ as follows
\begin{align*}
& C(A,B)  := \cdots \ra \Hom(A[1]^{\otimes k}, B[1]) \stackrel{\delta}{\ra} \Hom(A[1]^{\otimes k-1}, B[1]) \ra \cdots \ra B[1] \ra 0 \\
& \delta (\phi_k) (a_1,\ldots,a_{k-1}) := \sum_{j=0}^{k-1} (-1)^{|\phi_k|'+|a_1|'+\cdots+|a_{j}|'}\phi_k(a_1,\ldots,a_j,\mu_0^A,a_{j+1},\ldots,a_{k-1}).
\end{align*}
One verifies that $\delta^2=0$. We shall denote its cohomology by $D^k(A,B)$ where $k$ is the tensor degree. Next, using the operators $\mu_1^A$ and $\mu_1^B$, we define another operator $d$ as
\begin{align*}
& d: \Hom (A[1]^{\otimes k}, B[1]) \ra \Hom(A[1]^{\otimes k}, B[1]) \\
& d(\phi_k)(a_1,\ldots,a_k):=\mu_1^B\phi_k(a_1,\ldots,a_k)- \sum_{j=1}^k (-1)^{|\phi_k|'+\star}\phi_k(a_1,\ldots,\mu_1^A(a_j),\ldots,a_k)
\end{align*}
 where $\star= |a_1|'+\cdots+|a_{j-1}|'$. One can readily verify that $d\delta+\delta d=0$. Thus $d$ induces a map on the $\delta$-cohomology, i.e. we obtain maps
\[ d: D^k(A,B) \ra D^k(A,B), \;\; \forall k\geq 0.\]
In complete generality we don't have $d^2=0$. However, we have the following
\begin{lem}
If there exists a $R$-linear map $f: A \ra B$ such that $f(\mu_0^A)=\mu_0^B$, then  the composition $d^2: D^k(A,B) \ra D^{k}(A,B)$ equals  zero.
\end{lem}
\begin{proof}
Choose any such $f$. Given $[\phi_k]\in D^k(A,B)$, define $\psi\in \Hom(A[1]^{k+1}, B[1])$ as
\[ \psi := (-1)^{|\phi_k|'}\sum_{i=0}^{k-1} \phi_k(\id^i\otimes \mu_2^A \otimes \id^{k-1-i})+\mu_2^B(\phi_k\otimes f)+\mu_2^B(f\otimes \phi_k).\]
One can verify that we have $d^2(\phi_k)=\delta \psi$. Hence $d^2=0$ in $\delta$-cohomology.
\end{proof}

%\medskip

\begin{defi}\label{defi:obs}
Under the assumptions for $A$ and $B$ in the above lemma, we set the $k$-th obstruction space $H^k(A,B)$ to be the degree one cohomology of the complex $\big(D^k(A,B),d\big)$, i.e. 
\[ H^k(A,B) := H^1\big( D^k(A,B), d\big).\]
\end{defi}

\subsection{Obstruction classes.}\label{susbsec:obs} 
Let $f_j: A[1]^{\otimes j} \ra B[1] \;\; (j=1,\ldots,n)$ be $n$ multi-linear maps of cohomological degree zero, such that the following conditions hold:
\begin{itemize}
\item[(i)] The $A_\infty$ homomorphism axiom holds up to $(n-1)$ inputs, i.e. we have
\[ \sum_{j\geq 0, i_1,\cdots,i_j \geq 1 \atop i_1+\cdots+i_j=k}   \mu^B_j (f_{i_1}\otimes\cdots\otimes f_{i_j}) =\sum_{ r\geq 0,s\geq 0,t\geq 0\atop r+s+t=k} f_{r+t+1}(\id^{\otimes r} \otimes \mu^A_s\otimes \id^{\otimes t})\]
for all $0\leq k\leq n-1$.
\item[(ii)] In the case with $n$ inputs, we require that
\begin{equation}\label{eq:anhomo}
 \sum_{ j\geq 0, i_1,\cdots,i_j\geq 1 \atop i_1+\cdots+i_j=n}   \mu^B_j (f_{i_1}\otimes\cdots\otimes f_{i_j}) - \sum_{r\geq 0, t\geq 0, s\geq 1\atop r+s+t=n} f_{r+t+1}(\id^{\otimes r} \otimes \mu^A_s\otimes \id^{\otimes t})
\end{equation}
is $\delta$-exact, i.e. it lies in the image of $\delta: C^{n+1}(A,B) \ra C^n(A,B)$.
\end{itemize}
Given such a collection of maps $f_j: A[1]^{\otimes j} \ra B[1] \;\; (j=1,\ldots,n)$, we define its obstruction class as follows. Choose any $f_{n+1}'$ of cohomological degree zero such that $\delta f_{n+1}'$ equals the expression in (\ref{eq:anhomo}). Then we set 
\[ {\sf obs}_{n+1} :=   \sum_{ j\geq 2, i_1,\cdots,i_j\geq 1 \atop i_1+\cdots+i_j=n+1}   \mu^B_j (f_{i_1}\otimes\cdots\otimes f_{i_j}) - \sum_{r\geq 0, t\geq 0, s\geq 2\atop r+s+t=n+1} f_{r+t+1}(\id^{\otimes r} \otimes \mu^A_s\otimes \id^{\otimes t})+df_{n+1}'\]

\begin{lem}\label{lem:well-defined}
The above expression $ {\sf obs}_{n+1}$ is $\delta$-closed. Furthermore, $d  ({\sf obs}_{n+1})$ is $\delta$-exact. Thus it represents a well-defined class which we denote by $\mathfrak{o} \big( (f_j)_{j=1}^n\big)\in H^{n+1}(A,B)$. This class is independent of the choice of $f'_{n+1}$.
\end{lem}
\begin{proof}
Denote by $F_n$ the extension of $\  \sum_{j=1}^n f_j$ as a coalgebra map $T(A[1]) \to T(B[1])$. Similarly, denote by $\widetilde{\mu}_k$ the extension of $\mu_k$ as coderivations on the tensor coalgebra. Condition $(i)$ implies that we have
\[ \widetilde{\mu}^B F_n = F_n \widetilde{\mu}^A: A[1]^N \ra B[1]^M, \;\; \forall -1\leq N-M\leq n-2.\]
Observe also that $\delta(\phi)=(-1)^{|\phi|'}\phi\widetilde{\mu}_0$. Then we compute
\begin{align*}
&\delta({\sf obs}_{n+1}) = -\widetilde{\mu}_{\neq 1} F_n \widetilde{\mu}_0 + F_n\widetilde{\mu}_{\geq 2} \widetilde{\mu}_0 - \widetilde{\mu}_1 f_{n+1}' \widetilde{\mu}_0 + f'_{n+1}\widetilde{\mu}_1\widetilde{\mu}_0=\\
& ( -\widetilde{\mu}_{\neq 1} \widetilde{\mu} F_n +\widetilde{\mu}_{\neq 1} F_n \widetilde{\mu}_{\geq 1} ) + F_n\widetilde{\mu}_{\geq 2}\widetilde{\mu}_0 - (\widetilde{\mu}_1\widetilde{\mu}F_n + \widetilde{\mu}_1F_n \widetilde{\mu}_{\geq 1}) + (-\widetilde{\mu}F_n\widetilde{\mu}_1 + F_n\widetilde{\mu}_{\geq 1} \widetilde{\mu}_1)
\end{align*}
The two terms $-\widetilde{\mu}_{\neq 1} \widetilde{\mu} F_n$ and $-\widetilde{\mu}_1\widetilde{\mu}F_n$ combine to give zero since $\widetilde{\mu}\widetilde{\mu}=0$. The following three terms give \[\widetilde{\mu}_{\neq 1} F_n \widetilde{\mu}_{\geq 1} + \widetilde{\mu}_{1} F_n \widetilde{\mu}_{\geq 1} - \widetilde{\mu}F_n\widetilde{\mu}_1=\widetilde{\mu}F_n\widetilde{\mu}_{\geq 2}\]
Since there are $n$ inputs, after applying $\widetilde{\mu}_{\geq 2}$, we are left with at most $n-1$ inputs to apply $\widetilde{\mu}F_n$. In this case, we may use the commutativity $\widetilde{\mu}F_n= F_n\widetilde{\mu}$, i.e. we have the above three terms sum up to
 \[\widetilde{\mu}_{\neq 1} F_n \widetilde{\mu}_{\geq 1} + \widetilde{\mu}_{1} F_n \widetilde{\mu}_{\geq 1} - \widetilde{\mu}F_n\widetilde{\mu}_1=F_n\widetilde{\mu}\widetilde{\mu}_{\geq 2}=-F_n\widetilde{\mu}\widetilde{\mu}_{0}-F_n\widetilde{\mu}\widetilde{\mu}_{1}\]
The last equality follows from $\widetilde{\mu}\widetilde{\mu}=0$. Putting these back into the calculation of $\delta({\sf obs}_{n+1})$, we obtain 
 \begin{align*}
\delta({\sf obs}_{n+1}) & = F_n\widetilde{\mu}_{\geq 2}\widetilde{\mu}_0  + F_n\widetilde{\mu}_{\geq 1} \widetilde{\mu}_1-F_n\widetilde{\mu}\widetilde{\mu}_{0}-F_n\widetilde{\mu}\widetilde{\mu}_{1}
=& -F_n\widetilde{\mu}_1\widetilde{\mu}_0-F_n\widetilde{\mu}_0\widetilde{\mu}_1 =0
\end{align*}
Here we have used the $A_\infty$ relation that $\widetilde{\mu}_0\widetilde{\mu}_0=0$ and $\widetilde{\mu}_0\widetilde{\mu}_1+\widetilde{\mu}_1\widetilde{\mu}_0=0$.
Next, we prove that $d({\sf obs}_{n+1})$ is $\delta$-exact. We use the notation $F'_{n+1}: T(A[1]) \ra T(B[1])$ to denote the extension of $f_1,\cdots,f_n,f_{n+1}'$ to the tensor coalgebra. By definition of $f'_{n+1}$ we have that 
\[ \widetilde{\mu} F'_{n+1} = F'_{n+1}\widetilde{\mu}: A[1]^N \ra B[1]^M, \;\; \forall -1\leq N-M \leq n-1.\]
Using the notation $F_{n+1}'$, we may write the obstruction as
\[ {\sf obs}_{n+1}= \widetilde{\mu} F'_{n+1} - F'_{n+1}\widetilde{\mu}_{\geq 1}: A[1]^{n+1} \ra B[1].\]
Applying the operator $d$ to it yields
\[d( {\sf obs}_{n+1}) = \widetilde{\mu}_1\widetilde{\mu} F'_{n+1} - \widetilde{\mu}_1F'_{n+1}\widetilde{\mu}_{\geq 1}+\widetilde{\mu} F'_{n+1}\widetilde{\mu}_1-F'_{n+1}\widetilde{\mu}_{\geq 1}\widetilde{\mu}_1 .\]
The first and the third terms give
\begin{align*}
\widetilde{\mu}_1\widetilde{\mu} F'_{n+1} +\widetilde{\mu} F'_{n+1}\widetilde{\mu}_1& = -\widetilde{\mu}_{\geq 2} \widetilde{\mu} F'_{n+1} +\widetilde{\mu} F'_{n+1}\widetilde{\mu}_1\\
&=  -\widetilde{\mu}_{\geq 2} F'_{n+1} \widetilde{\mu} +\widetilde{\mu} F'_{n+1}\widetilde{\mu}_1\\
&= -\widetilde{\mu}_{\geq 2} F'_{n+1}\widetilde{\mu}_0 -\widetilde{\mu}_{\geq 2} F_n \widetilde{\mu}_{\geq 1}  +\widetilde{\mu} F_{n}\widetilde{\mu}_1+ \widetilde{\mu}_1 f'_{n+1}\widetilde{\mu}_1\\
&=  -\widetilde{\mu}_{\geq 2} F'_{n+1}\widetilde{\mu}_0 -\widetilde{\mu}_{\geq 2} F_n \widetilde{\mu}_{\geq 2}   +\widetilde{\mu}_1 f'_{n+1}\widetilde{\mu}_1\\
\end{align*}
Similarly, we have
\begin{align*}
- \widetilde{\mu}_1F'_{n+1}\widetilde{\mu}_{\geq 1}-F'_{n+1}\widetilde{\mu}_{\geq 1}\widetilde{\mu}_1 &= - \widetilde{\mu}_1F'_{n+1}\widetilde{\mu}_{\geq 1}-F'_{n+1}\widetilde{\mu}\widetilde{\mu}_1\\
&= - \widetilde{\mu}_1F'_{n+1}\widetilde{\mu}_{\geq 1}+F'_{n+1}\widetilde{\mu}\widetilde{\mu}_0+F'_{n+1}\widetilde{\mu}\widetilde{\mu}_{\geq 2}\\
&= - \widetilde{\mu}_1F'_{n+1}\widetilde{\mu}_{\geq 1}+F'_{n+1}\widetilde{\mu}\widetilde{\mu}_0+\widetilde{\mu}F_{n}\widetilde{\mu}_{\geq 2}\\
&= F'_{n+1}\widetilde{\mu}\widetilde{\mu}_0+\widetilde{\mu}F_{n}\widetilde{\mu}_{\geq 2}- \widetilde{\mu}_1F_{n}\widetilde{\mu}_{\geq 2}- \widetilde{\mu}_1f'_{n+1}\widetilde{\mu}_{1}\\
&= F'_{n+1}\widetilde{\mu}\widetilde{\mu}_0+\widetilde{\mu}_{\geq 2}F_{n}\widetilde{\mu}_{\geq 2}- \widetilde{\mu}_1f'_{n+1}\widetilde{\mu}_{1}
\end{align*}
Adding the two equations together yields the desired formula
\[ d( {\sf obs}_{n+1}) =-\widetilde{\mu}_{\geq 2} F'_{n+1}\widetilde{\mu}_0+F'_{n+1}\widetilde{\mu}\widetilde{\mu}_0 = \delta\big( \widetilde{\mu}_{\geq 2} F'_{n+1}-F'_{n+1}\widetilde{\mu}\big)\]
Finally, to see that the class $\mathfrak{o}\big( (f_j)_{j=1}^n\big)=[{\sf obs}_{n+1}]$ is independent of $f'_{n+1}$, let $f''_{n+1}$ be another such map. Then we have the two obstructions differ by $d(f'_{n+1}-f''_{n+1})$ with $\delta(f'_{n+1}-f''_{n+1})=0$, this proves the two obstruction classes are equal.
\end{proof}

\subsection{$A_{(n)}$ homomorphisms.} \label{subsec:Ahomotopy}
In the following, we shall refer to a collection of maps $(f_j: A[1]^{\otimes j} \ra B[1])_{j=1}^n$ satisfying conditions $(i)$, $(ii)$ in Subsection~\ref{susbsec:obs} as an $A_{(n)}$ homomorphism from $A$ to $B$.  Obviously, an $A_{(n)}$ homomorphism is also an $A_{(k)}$ homomorphism, for any $k\leq n$. Just as in the case of the usual $A_\infty$ homomorphisms, one can compose $A_{(n)}$ homomorphisms, using the formula
\begin{equation}\label{eq:composition}
(g\circ f)_j:= \sum_{l,i_1,\ldots,i_l\geq 1\atop i_1+\cdots+i_l=j} g_l(f_{i_1}\otimes\cdots\otimes f_{i_l}),
\end{equation}
for $j=1,\ldots,n$. If we denote by $f_{n+1}'$ (and $g_{n+1}'$) a map such that $\delta f_{n+1}'$ equals the expression in (\ref{eq:anhomo}) then we can define
\[(g\circ f)_{n+1}':= g_{n+1}'(f_{1}\otimes\cdots\otimes f_{1} )+  g_1(f_{n+1}')+\sum_{l,i_1,\ldots,i_l\geq 1\atop i_1+\cdots+i_l=n+1} g_l(f_{i_1}\otimes\cdots\otimes f_{i_l}),\]
Please note that this composition is strictly associative.

The simple but crucial observation is that a $A_{(n)}$ homomorphism lifts to an $A_{(n+1)}$ homomorphism if and only if its obstruction class $\mathfrak{o}\big(  (f_j)_{j=1}^n\big)$ vanishes.

Next we recall the notion of homotopy between $A_{(n)}$ homomorphisms. Let $\Omega_{[0,1]}^*$ be piece-wise polynomial differential forms on the interval $[0,1]$ (as in \cite[Definition 4.2.9]{FOOO}). This is an unital dg-algebra, therefore given an $A_\infty$-algebra $B$ we can easily define the tensor product $B \otimes \Omega_{[0,1]}^*$ (see \cite{Amo} for details). Moreover there are naive (also called linear) maps of $A_\infty$-algebras ${\sf ev}_0, {\sf ev}_1: B\otimes \Omega_{[0,1]}^* \to B$, given by ``evaluating at $t=0,1$". 
Two $A_{(n)}$ homomorphisms $f=(f_1,\ldots,f_n), g=(g_1,\ldots,g_n): A \to B$ are called homotopic, denoted by $f\cong g$, if there exists an $A_{(n)}$ homomorphism 
\[ F=(F_1,\ldots,F_n) : A \to B\otimes \Omega_{[0,1]}^*\]
such that ${\sf ev}_0\circ F = f$ and ${\sf ev}_1\circ F= g$. The following properties are standard:
\begin{itemize}
\item Homotopy relation ``$\cong$ " between $A_{(n)}$ morphisms is an equivalence relation.
\item If $h: A' \ra A$ is another $A_{(n)}$ morphism and $f\cong g$, then $f\circ h\cong g\circ h$.
\item If $h: B \ra B'$  is another $A_{(n)}$ morphism and $f\cong g$, then $h\circ f\cong h\circ g$.
\item Two homotopic $A_{(n)}$ homomorphisms $(f_j)_{j=1}^n$ and $(g_j)_{j=1}^n$ from $A$ to $B$  have the same obstruction class, i.e. we have
\[ \mathfrak{o}\big( (f_j)_{j=1}^n \big) = \mathfrak{o}\big( (g_j)_{j=1}^n\big).\]
Indeed, it is clear that we have 
\begin{align*}
\mathfrak{o}\big( (f_j)_{j=1}^n \big) & = ({\sf ev}_0)_* \mathfrak{o}\big( (F_j)_{j=1}^n \big)\\
\mathfrak{o}\big( (g_j)_{j=1}^n \big) & = ({\sf ev}_1)_* \mathfrak{o}\big( (F_j)_{j=1}^n \big)
\end{align*}
Then observe that every cohomology class in $H^{n+1}(A,B\otimes  \Omega_{[0,1]}^*)$ can be represented by an element of the form $\phi\otimes \bone$ with $\phi\in C^{n+1}(A,B)$ and $\bone$ the constant function in $\Omega_{[0,1]}^*$, which clearly shows that applying the two evaluation maps $\ev_0$ and $\ev_1$ both yield $[\phi]\in H^{n+1}(A,B)$.
\end{itemize}

It is useful to spell out the definition of homotopy in the $n=1$ case.

\begin{lem}\label{lem:a1homo}
	Two $A_{(1)}$-homomorphisms $f_1, g_1: A \ra B$ are $A_{(1)}$-homotopic if and only if there is a map $H:A\ra B$ of degree $-1$, such that $H(\mu^A_0)=0$ and $f_1 - g_1 - \mu^B_1 H - H \mu^A_1$ is $\delta$-exact.
\end{lem}
\begin{proof}
	An $A_{(1)}$-homotopy gives a map $F_1: A \ra B\otimes \Omega_{[0,1]}^*$. If we write $F_1(a)= f_1^t(a)+(-1)^{|a|'} h_1^t(a)dt$ then the $A_{(1)}$-homomorphism equation for $F$ is equivalent to: $f_1^1=f_1$, $f_1^0=g_1$, $h_1^t(\mu_0^A)=0$ and $-\frac{d f_1^t}{dt}+ \mu_1^B h_1^t+ h_1^t\mu_1^A$ is $\delta$-exact. We obtain the desired equality by taking $H=\int_0^1 h_1^t dt$.
\end{proof}

Two $A_\infty$ algebras $A$ and $B$ are called $A_{(n)}$ homotopic, if there exist $A_{(n)}$ homomorphisms $f: A \to B$ and $g: B \to A$ such that $f\circ g$ and $g\circ f$ are both $A_{(n)}$ homotopic to the identity homomorphism.

\subsection{Homotopy invariance of obstruction theory.} The obstruction spaces are natural with respect to $A_{(1)}$ homomorphisms. Namely, fix two curved $A_\infty$ algebras $A$ and $B$. Assume that there exists an $R$-linear map $f: A\to B$ such that $f(\mu_0^A)=\mu_0^B$, so that the obstruction space $H^k(A,B)$ is defined. Let $h: B \to B'$ be an $A_{(1)}$ homomorphism. By definition, we have $hf(\mu^A_0)=h(\mu^B_0)=\mu_0^{B'}$, which shows that the obstruction space $H^k(A,B')$ is also defined. Furthermore, the morphism $h$ induces a push-forward map
\[ h_*: H^k(A,B) \to H^k(A,B'),\]
defined by $[\phi] \mapsto [h\circ \phi] $ where $\phi\in {\Hom}(A[1]^{\otimes k}, B[1])$ is a representative. Similarly, let $g: A'[1] \ra A[1]$ be an $A_{(1)}$ homomorphism. We may define the pull-back map
\[ g^*: H^k(A,B) \ra H^k(A',B),\]
by $[\phi]\mapsto [\phi\circ (g\otimes\cdots\otimes g)]$ where we used $k$-copies of $g$ in the tensor product. 

\begin{lem}\label{lem:hom_inv}
	Assume that $h:B \to B'$ and $g: A'\to A$ are both $A_{(1)}$ homotopy equivalences. Then both $h_*$ and $g^*$ are isomorphisms.
\end{lem}
\begin{proof}
	It follows from Lemma \ref{lem:a1homo} that	if $h_0$ and $h_1$ are $A_{(1)}$-homotopic, then $(h_0)_*=(h_1)_*$ and $(h_0)^*=(h_1)^*$. This together with the identities $(h_0 h_1)_*=(h_0)_* (h_1)_*$ and $(g_1 g_0)^*=(g_0)^* (g_1)^*$ immediately give the result.
\end{proof}

\subsection{Whitehead theorem for curved $A_\infty$-algebras.}

\begin{prop}\label{prop:torsor}
Let $f=(f_1,\ldots,f_n): A \to B$ be an $A_{(n)}$ homomorphism. Assume that $\mathfrak{o}\big( (f_j)_{j=1}^n \big)=0$. Denote by $\sL(f)$ the set of liftings of $f$ to an $A_{(n+1)}$ homomorphism modulo the homotopy equivalence relation. Then $\sL(f)$ carries a natural transitive action by the abelian group $H^0\big(D^{n+1}(A,B),d\big)$.
\end{prop}
\begin{proof}
Let $f_{n+1}$ be a lift of $f$ to an $A_{(n+1)}$ homomorphism. Denote by $[f_{n+1}]$ its equivalence class in $\sL(f)$. Let $\beta: A[1]^{\otimes n+1} \ra B[1]$ be a map representing an element $[\beta]\in H^0\big(D^{n+1}(A,B),d\big)$. We define the group action by the formula
\begin{equation}~\label{eq:action}
 [\beta].[f_{n+1}] := [f_{n+1}+\beta]
 \end{equation}
To see that the action is independent of the choice of $\beta$, let $\beta'$ be another representative of the class $[\beta]$. Thus, the difference $\beta'-\beta=d\alpha$ for some $\delta$-closed morphism $\alpha: A[1]^{\otimes n+1} \ra B[1]$. We may define a homotopy between the two extensions $(f_1,\ldots,f_{n+1}+\beta)$ and $(f_1,\ldots,f_{n+1}+\beta')$ by putting
\begin{align*}
F &: A \to B\otimes \Omega^*_{[0,1]}\\
F_k &= f_k, \;\; 1\leq k\leq n\\
F_{n+1} &=  f_{n+1}+t\cdot \beta'+(1-t)\cdot \beta+ \alpha \cdot dt
\end{align*}
This shows that the action map~\eqref{eq:action} is independent of the choice of $\beta$.

Similarly, assume that $f'_{n+1}$ is another representative of the lift class $[f_{n+1}]$, i.e. there exists a homotopy $H: A\to B\otimes \Omega^*_{[0,1]}$ between $(f_1,\ldots,f_n,f_{n+1})$ and $(f_1,\ldots,f_n,f'_{n+1})$. We simply change $H_{n+1}$ to $H_{n+1}+\beta$, which gives a homotopy between the two lifts $(f_1,\ldots,f_n,f_{n+1}+\beta)$ and $f_1,\ldots,f_n,f'_{n+1}+\beta)$. This verifies that the action map is well-defined.

Transitivity of the action map is clear: since any two lifts differ by some $\beta$ that would represent a class in $H^0\big(D^{n+1}(A,B),d\big)$.
\end{proof}

\begin{lem}\label{obs_composition}
	Let $f:A \to B$ be an $A_{(n)}$ homomorphism. Assume that $h:B \to B'$ and $g: A'\to A$ are both $A_{(n+1)}$ homomorphisms. Then we have
	\begin{align*}
		& \mathfrak{o}\Big( \big( (hf)_j\big)_{j=1}^n \Big) = (h_1)_* \mathfrak{o}\big( (f_j)_{j=1}^n \big)\\
		& \mathfrak{o}\Big(\big(fg)_j\big)_{j=1}^n \Big) = (g_1)^*\mathfrak{o}\big( (f_j)_{j=1}^n \big)
	\end{align*}
	Moreover, the natural map
	$ -\circ g: \sL(f) \to \sL(f\circ g)$
	given by composition of $A_{(n+1)}$ homomorphisms is a homomorphism of $H^0\big(D^{n+1}(A,B),d\big)$-modules. 
	Here the $H^0\big(D^{n+1}(A,B),d\big)$-module structure of $\sL(f\circ g)$ is via the group homomorphism $$(g_1)^*: H^0\big(D^{n+1}(A,B),d\big) \ra H^0\big(D^{n+1}(A',B),d\big).$$
\end{lem}
\begin{proof}
	The fist statements can be proved as in the uncurved case, see Theorem 4.5.1 in \cite{FOOO}. The second statement follows from the action map~\eqref{eq:action} and the formula for composition in (\ref{eq:composition}).
\end{proof}

%\medskip
\begin{thm}\label{thm:A1}
An $A_\infty$ homomorphism $f=(f_1,f_2,\ldots): A \to B$ between curved $A_\infty$ algebras is a homotopy equivalence if and only if the map $f_1$ is an $A_{(1)}$ homotopy equivalence.
\end{thm}

\begin{proof}
The only if part is trivial, so we prove the if part. Let $ g_1: B \to A$ be an $A_{(1)}$ homotopy inverse of $f_1: A \to B$. We argue by induction on $n$ that if we are given an $A_{(n)}$ homomorphism
\[ g:=(g_1,\ldots,g_n): B \to A\]
such that $g\circ f\cong \id$ as $A_{(n)}$ homomorphisms, then there exists $g_{n+1}: B[1]^{\otimes n+1} \ra A[1]$ that extends $g$ to an $A_{(n+1)}$ homomorphism $\widetilde{g}=(g_1,\ldots,g_n,g_{n+1})$ such that $\widetilde{g}\circ f\cong \id$ as $A_{(n+1)}$ homomorphisms.

We first argue that $\mathfrak{o}(g)=0$. Using the homotopy invariance of obstruction class, we have
\[ f^*\mathfrak{o}(g)=\mathfrak{o}(g\circ f) = \mathfrak{o}(\id)=0\]
But $f$ is an $A_{(1)}$ homotopy equivalence, thus $f^*$ is an isomorphism, which shows that $\mathfrak{o}(g)=0$. Similarly, one can argue that if $H: A\ra A\otimes \Omega_{[0,1]}^*$ is a $A_{(n)}$ homotopy between $\id$ and $g\circ f$, then we also have $\mathfrak{o}(H)=0$. 

Now consider the following diagram of maps, provided by Proposition~\ref{prop:torsor} and Lemma~\ref{obs_composition},
\[\begin{CD}
\sL(H)   @> (\ev_0)_*>>  \sL(\id_A) \\
@V(\ev_1)_*VV                 @. \\
\sL(g\circ f) @< -\circ f << \sL(g)
\end{CD}\]
Observe that the upper-right corner admits a canonical lift by $\id_A$. We claim that there exists a lift $\widetilde{H}$ of $H$ such that
\[ (\ev_0)_* (\widetilde{H})=\id_A\]
Indeed, let $\widetilde{H}'$ be any lift of $H$. By the transitivity of the action map, there exists an element $\beta\in H^0\big(D^{n+1}(A,A),d\big)$ such that
\[ \beta. \big(  (\ev_0)_* (\widetilde{H}') \big) = \id_A\]
Since $\ev_0$ is a homotopy equivalence, there exists $\gamma\in H^0\big(D^{n+1}(A,A\otimes\Omega_{[0,1]}^*),d\big)$ such that $(\ev_0)_*\gamma=\beta$.  Using Lemma~\ref{obs_composition} we obtain
\[ (\ev_0)_*\big( \gamma.\widetilde{H}' \big)=\beta. \big(  (\ev_0)_* (\widetilde{H}') \big) = \id_A\]
We set $\widetilde{H}:= \gamma.\widetilde{H}'$. By the same argument, one can show that there exists a lift $\widetilde{g}$ of $g$ such that 
\[ (\ev_1)_* (\widetilde{H})= \widetilde{g}\circ f\]
In conclusion, we obtained an $A_{(n+1)}$ homomorphism $\widetilde{g}: B\ra A$ such that $\widetilde{g} \circ f \cong \id_A$. 

Finally, we need to prove that $f \circ \widetilde{g} \cong \id_B$. Since $\widetilde{g}$ is also a weak equivalence, the conclusion above implies that there exists an $A_{(n+1)}$ homomorphism $f': A \ra B$ extending $(f_1,\ldots,f_n)$ such that $f'\circ \widetilde{g}\cong \id_B$. Thus we have
\[ f\circ \widetilde{g} \cong f'\circ \widetilde{g} \circ f\circ \widetilde{g} \cong f'\circ \widetilde{g} \cong \id_B\]
which finishes the proof.
\end{proof}

\begin{rmk}\label{rmk:whitehead}
Observe that in the uncurved case, according to Lemma~\ref{lem:a1homo} our notion of $A_{(1)}$ homotopy between morphisms of chain complexes agrees with the usual one. Furthermore, if we are over a field, quasi-isomorphic chain complexes are in fact homotopy equivalent. Thus, the above theorem easily implies the usual Whitehead theorem of uncurved $A_\infty$ algebras over a field which states that a quasi-isomorphism between uncurved $A_\infty$ algebras over a field is in fact a homotopy equivalence.
\end{rmk}

\subsection{Curved $L_\infty$ algebras} The previous discussion and results have direct analogues in the $L_\infty$ setting. Let $A$ and $B$ be two curved $L_\infty$ algebras. In this case, we set the $\delta$-complex $C(A,B)$ as
\begin{align*}
& C(A,B)  := \cdots \ra \Hom(\sym^k A[1], B[1]) \stackrel{\delta}{\ra} \Hom(\sym^{k-1}A[1], B[1]) \ra \cdots \ra B[1] \ra 0 \\
& \delta (\phi_k) (a_1\cdots a_{k-1}) := (-1)^{|\phi_k|'}\cdot \phi_k(\mu_0^A\cdot a_1\cdots a_{k-1})
\end{align*}
One verifies that $\delta^2=0$. As before, we denote its cohomology by $D^k(A,B)$ where $k$ is the tensor degree. Using the operators $\mu_1^A$ and $\mu_1^B$, we define another operator $d$ by
\begin{align*}
& d: \Hom (\sym^k A[1], B[1]) \ra \Hom(\sym^{k}A[1], B[1]) \\
& d(\phi_k)(a_1\cdots a_k):=\mu_1^B\phi_k(a_1\cdots a_k)- \sum_{j=1}^k (-1)^{\star}\phi_k(a_1\cdots \mu_1^A(a_j)\cdots a_k)
\end{align*}
with $\star=|\phi_k|'+|a_1|'+\cdots+|a_{j-1}|'$.
Here $\sym$ stands for the graded symmetric algebra.

 If there exists a $R$-linear map $f: A \ra B$ such that $f(\mu_0^A)=\mu_0^B$, we set the $k$-th obstruction space $H^k(A,B)$ to be $H^k(A,B) := H^1\big( D^k(A,B), d\big)$. 

Let $f_j: \sym^j A[1] \ra B[1] \;\; (j=1,\ldots,n)$ be $n$ multi-linear maps of cohomological degree zero. We call the sequence $(f_1,\ldots,f_n)$ an $L_{(n)}$ morphism if the following conditions hold:
\begin{itemize}
\item[(i)] The $L_\infty$ homomorphism axiom holds up to $(n-1)$ inputs, i.e. for all $0\leq m\leq n-1$ we have
\begin{align*}
 &\sum_k \frac{1}{k!}\sum_{\sigma}\epsilon_\sigma\cdot \mu^B_k\big(f_{i_1}(a_{\sigma(1)}\cdots) \cdots f_{i_k}(\cdots a_{\sigma(m)})\big)\\
 =&\sum_{r\geq 0}\sum_\tau \epsilon_\tau\cdot f_{m-r+1}\big(\mu^A_r(a_{\tau(1)}\cdots a_{\tau(r)})\cdots a_{\tau(m)}\big)\end{align*}
where $\sigma$ is a $(i_1,\cdots,i_k)$ type shuffle, and $\tau$ is a $(r,n-r)$ type shuffle, and $\epsilon_\sigma$ and $\epsilon_\tau$ are Koszul signs associated with these permutations.
\item[(ii)] In the case with $n$ inputs, we require that
\begin{equation}\label{eq:lnhom}
 \sum_{k,\sigma} \frac{1}{k!} \mu^B_k (f_{i_1}\otimes\cdots\otimes f_{i_k})\Sh_\sigma - \sum_{r\geq 1,\tau} f_{n-r+1}(\mu^A_r\otimes \id^{\otimes n-r})\Sh_\tau
\end{equation}
is $\delta$-exact, i.e. it lies in the image of $\delta: C^{n+1}(A,B) \ra C^n(A,B)$. Here $\sigma$, $\tau$ are as above and $\Sh_\sigma$ is the map that permuts the inputs according to the shuffle $\sigma$.
\end{itemize}
Given an $L_{(n)}$ homomorphism $(f_1,\ldots,f_n): A\ra B$, we define its obstruction class as follows. Choose any $f_{n+1}'$ of cohomological degree zero such that $\delta f_{n+1}'$ equals the expression in (\ref{eq:lnhom}). Then we set 
\begin{align*}
 &{\sf obs}_{n+1}:= \\
 & \sum_{k\geq 2,\sigma\atop i_1+\cdots+i_k=n+1} \frac{1}{k!} \mu^B_k (f_{i_1}\otimes\cdots\otimes f_{i_k})\Sh_\sigma - \sum_{r\geq 2,\tau} f_{n-r+2}(\mu^A_r\otimes \id^{\otimes n+1-r})\Sh_\tau+df_{n+1}'
 \end{align*}
and define the obstruction class by
\[ \mathfrak{o}\big( (f_j)_{j=1}^n \big) = [ {\sf obs}_{n+1} ] \in H^{n+1}(A,B)\]
Again, one can verify (similar to Lemma~\ref{lem:well-defined}) that this class is well-defined and independent of the choice of $f'_{n+1}$. 

The formal properties of the obstruction theory still holds in the $L_\infty$ case, with which we deduce the following result for curved $L_\infty$ algebras. Again, in the uncurved case and over a field, this result immediately implies the classical Whitehead theorem of $L_\infty$ algebras that quasi-isomorphisms are also homotopy equivalences.

\begin{thm}
An $L_\infty$ homomorphism $f=(f_1,f_2,\ldots): A \to B$ between curved $L_\infty$ algebras is a homotopy equivalence if and only if $f_1$ is an $L_{(1)}$ homotopy equivalence.
\end{thm}

\section{Homotopy transfer of curved algebras}~\label{sec:transfer}

In this section, we prove a curved version of the homological perturbation lemma. This works for both $A_\infty$ and $L_\infty$ algebras in the presence of a new version of homotopy retraction data in the curved case. This is used to construct minimal charts of $L_\infty$ spaces  in Subsection~\ref{subsec:minimalchart}.

\subsection{Curved homotopy retraction data} Let us consider the following situation: we are given a curved $A_\infty$-algebra $(A,m_k)$, a graded vector space $V$, $R$-linear degree zero maps $i: V \ra A$ and $p: A \ra V$ and a $R$-linear map $H: A \ra A$ of degree $-1$. Assume there is $C \in V$ of degree 2 satisfying $i(C)= m_0$ and moreover:
\begin{align}
	m_{1}H+H m_{1} &= ip -\id_{A} - H m_1^2 H,\label{homotopy}\\
	p m_1 H =0, \ & \ H m_1 i=0.\label{homotopy_2}
\end{align}

We have the following
\begin{thm}\label{teorema B1}
	In the situation described, there is a curved $A_\infty$-algebra structure on $V$ with $\mu_{0}=C$ and $\mu_1=p m_1 i$. Moreover there is an $A_\infty$-homomorphism $\varphi:(V,\mu_k)\ra(A,m_k)$ with $\varphi_{1}=i$.
\end{thm}

A common application of this theorem is to construct ``minimal" algebras. In that case, we have the side conditions $Hi=pH=0$ and $p i=\id$. In the presence of these extra conditions (\ref{homotopy_2}) follows from (\ref{homotopy}).

Before we go into the proof we describe the maps $\mu_k$, for $k\geq 2$:
$$\mu_k=\sum_{T \in \Gamma_k}\mu_{T},$$
where $\Gamma_k$ is the set of rooted stable planar trees with $k$-leaves.
\begin{figure}[h]
	\begin{center}
		\setlength{\unitlength}{2pt}
		\begin{picture}(100,70)(0,0)
		%\put(0,10){\circle*{2}}
		\put(20,26){\circle*{2}}
		\put(50,50){\circle*{2}}
		\put(80,26){\circle*{2}}
		%\put(56.2,26){\circle*{2}}
		
		\linethickness{0.3mm}
		\put(50,50){\line(0,1){20}}
		\linethickness{0.3mm}
		\multiput(0,10)(0.15,0.12){333}{\line(1,0){0.15}}
		\linethickness{0.3mm}
		\put(20,10){\line(0,1){16}}
		\linethickness{0.3mm}
		\multiput(20,26)(0.16,-0.12){125}{\line(1,0){0.16}}
		\linethickness{0.3mm}
		\multiput(50,50)(0.12,-0.48){83}{\line(0,-1){0.48}}
		\linethickness{0.3mm}
		\multiput(50,50)(0.15,-0.12){333}{\line(1,0){0.15}}
		\linethickness{0.3mm}
		\put(80,10){\line(0,1){16}}
		\linethickness{0.3mm}
		%\put(51,40){\line(1,0){3}}
		\linethickness{0.3mm}
		\multiput(33,38)(0.12,-0.12){15}{\line(1,0){0.12}}
		\linethickness{0.3mm}
		%\multiput(7,17)(0.12,-0.12){15}{\line(1,0){0.12}}
		\linethickness{0.3mm}
		\multiput(66,36)(0.12,0.12){15}{\line(1,0){0.12}}
		
		\put(45,70){\makebox(0,0)[cc]{$p$}}
		\put(20,5){\makebox(0,0)[cc]{$i$}}
		\put(40,5){\makebox(0,0)[cc]{$i$}}
		\put(60,5){\makebox(0,0)[cc]{$i$}}
		\put(80,5){\makebox(0,0)[cc]{$i$}}
		\put(100,5){\makebox(0,0)[cc]{$i$}}
		\put(70,40){\makebox(0,0)[cc]{H}}
		%\put(5,20){\makebox(0,0)[cc]{H}}
		\put(30,40){\makebox(0,0)[cc]{H}}
		%\put(49,40){\makebox(0,0)[cc]{H}}
		\put(85,29){\makebox(0,0)[cc]{$m_{2}$}}
		\put(56,50){\makebox(0,0)[cc]{$m_3$}}
		%\put(50,29){\makebox(0,0)[cc]{$m_{1}$}}
		\put(0,5){\makebox(0,0)[cc]{$i$}}
		\put(17,29){\makebox(0,0)[cc]{$m_{3}$}}
		\end{picture}
	\end{center}
	\caption{Example of an element of $\Gamma_6$.}
\end{figure}
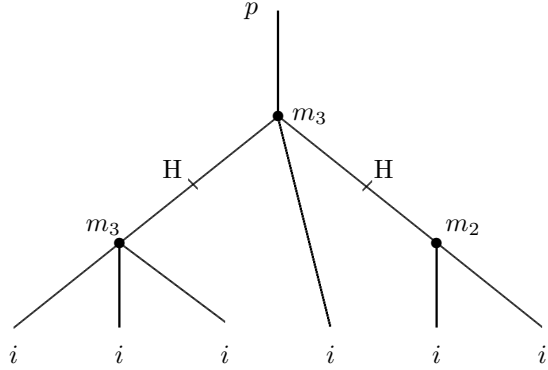

We use $T$ as a flow chart to define a map
$\mu_T: V^{\otimes k}\ra V.$
We assign to each $v\in V(T)$ the map $m_{val(v)}$; to the internal edges we assign $H$; and finally we assign $p$ to the root and $i$ to the leaves. For examples, the tree in Figure 1 gives the map
\begin{align}
	\mu_T(u_1,u_2,u_3,u_4,&u_5, u_6)=\nonumber\\
	=&p\circ m_{3}(H\circ m_{3}(i(u_1),i(u_2),i(u_3)),i(u_4),H\circ m_{2}(i(u_5),i(u_6))).\nonumber
	\end{align}
	
We would like to point out that these are exactly the same formulas as in the uncurved case. In particular, the $m_0$ term plays no role in the formulas for $\mu_k, \ k \geq 2$.

To prove these maps define an $A_\infty$-algebra
we will need the following auxiliary maps. Let $T\in\Gamma_k$, denote by $E(T)$ the set of edges of $T$ and by $e(T)$ the set of internal edges of $T$. For each $T\in \Gamma_k$, we define $\overline T$ as the  tree $T$ with one additional vertex in each internal edge of $T$. Given $e\in E(\overline T)$ we define $\hat\mu_{\overline{T}, e}$ in the same way as $\mu_{ T}$ with the extra assignment of $m_1$ to the edge $e$. Given $e\in e(T)$ we define $\mu^\Pi_{T,e}$, $\mu^{id}_{T,e}$ and $\mu^\gamma_{T,e}$ in the same way as $\mu_{ T}$, but with $\Pi=ip$ (respectively $id$ and $\gamma:=H m_1^2 H$) assigned to $e$ instead of $H$.

\begin{proof}[Proof of Theorem \ref{teorema B1}] 
	One can easily check the first two $A_\infty$ equations for $\mu_k$ using equations (\ref{homotopy}, \ref{homotopy_2}).	For two or more inputs we define
	$$\hat\mu_{k,\beta}(u_1,\ldots,u_k)=\sum_{\substack{T\in\Gamma_k\\e\in E({T})}} (-1)^{\vert {T}_e\vert}\hat\mu_{\overline{T},e}(u_1,\ldots,u_k)$$
	where $\vert {T}_e\vert=\sum_{i=1}^{m_e}\vert u_i\vert'$ with $m_e$ defined as the smallest $1\leq j\leq k$ such that the path from the $i^{th}$ leaf to the root does not include $e$, for all $i<j$.  Then given $e\in e(T)$, denote by $E_-$ and $E_+$ the edges of $\overline{T}$ contained in $e$. The equation (\ref{homotopy}) implies
	$$\hat\mu_{\overline{T},E_-}+\hat\mu_{\overline{T},E_+}=\mu^\Pi_{T,e}-\mu^{id}_{T,e} - \mu^{\gamma}_{T,e}.$$ 
	Therefore
	\begin{align}\label{eq:pert_1}
		\hat\mu_{k}=\sum_{\substack{T\in\Gamma_k\\ e\in E( T)\setminus e(T)}} (-1)^{\vert T_e\vert}\hat\mu_{\overline{T},e}+\sum_{\substack{T\in\Gamma_k\\ e\in e(T)}}(-1)^{\vert T_e\vert}\big(\mu^\Pi_{T,e}-\mu^{id}_{T,e}-\mu^{\gamma}_{T,e}\big).
	\end{align}
	On the other hand,
	$$\hat\mu_{k}=\sum_{T\in\Gamma_k}\sum_{v\in V(T)}\sum_{\substack{e\in E(\overline T)\\v\in \partial e}}(-1)^{\vert T_e\vert}\hat\mu_{\overline{T},e},$$
	and the $A_\infty$-equation implies that
	\begin{equation}\label{eq:pert_2}
	\hat\mu_{k}=-\sum_{T\in\Gamma_k}\sum_{v\in V(T)}\sum_{\substack{S\in\Gamma_k\\e\in e(S)\\S/e= T}}(-1)^{\vert S_e\vert}\mu^{id}_{S,e} - \sum_{\substack{S\in\Gamma_{k+1}\\S/i= T}}(-1)^{\vert S_e\vert}\mu_{S}(...u_{i-1}, \mu_0, u_{i}...).
	\end{equation}
	Here $S/e$ is the tree obtained from $S$ by collapsing the edge $e$ and $S/i$ is the tree obtained by deleting the $i$-th leaf of $S$.
	Putting (\ref{eq:pert_1}) and (\ref{eq:pert_2}) together we conclude
\begin{align}\label{eq:pert_sum}
\sum_{T\in\Gamma_k}\Big(\sum_{e \in e(T)}(-1)^{\vert T_e\vert}&(\mu^\Pi_{T,e}-\mu^{\gamma}_{T,e})+\hat\mu_{{T},r} +  \sum_{i=1}^{k} (-1)^{\vert T_e\vert}\hat\mu_{{T},l_i}\Big)\nonumber\\
& + \sum_{T\in\Gamma_k}\sum_{\substack{S\in\Gamma_{k+1}\\S/i=T}}(-1)^{\vert S_e\vert}\mu_{S}(...u_{i-1}, \mu_0, u_{i}...)=0,
\end{align}
where $r$ is the edge adjacent to the root and the $l_i$ are the edges adjacent to the leaves of $T$. It follows from the definition of $\mu_T$ that 
	$$\sum_{T\in\Gamma_k}\sum_{e \in e(T)}(-1)^{\vert T_e\vert}\mu^\Pi_{T,e}= \sum_{\substack{k_1\neq 0,1\\k_2\neq 1}}(-1)^{*}\mu_{k_2}(u_1,\ldots,\mu_{k_1}(u_{i+1},\ldots,u_{i+k_1}),\ldots,u_k).$$
Equations (\ref{homotopy}) and (\ref{homotopy_2}) imply $p m_1= \mu_1 p - p m_1^2 H$. This combined with the $A_\infty$ equation gives
$$\hat\mu_{T,r}= \mu_1 \mu_T + \mu_{C_2 \circ_2 T}(\mu_0,\ldots)+ (-1)^* \mu_{C_2 \circ_1 T}(\ldots,\mu_0),$$
where $C_2$ is the unique tree with two leaves and $C_2 \circ_i T$ is the tree obtained by grafting the root of $T$ tot he $i$-th leaf of $C_2$. 

Analogously, the identity $m_1 i= i \mu_1 - H m_1^2 i$ implies
\begin{align*}
\sum_{i=1}^{k} (-1)^{\vert T_e\vert}\hat\mu_{{T},l_i}   = & \sum_i \mu_T(..., \mu_1(u_i)),...) + \sum_i (-1)^* \mu_{T\circ_i C_2}(..., \mu_0,u_i,...)\\
& + (-1)^* \mu_{T\circ_i C_2}(...,u_i, \mu_0,...)
\end{align*}
Finally, using the fact that $-\gamma(u)=H m_2(m_0, H(u))+(-1)^{|u|}m_2(H(u), m_0)$ we have
$$(-1)^{\vert T_e\vert+1}\mu^{\gamma}_{T,e}= (-1)^* \mu_{T_1 \circ_i C_2 \circ_ 2 T_2}(..., \mu_0,u_i,...)+ (-1)^* \mu_{T_1 \circ_i C_2 \circ_ 1 T_2}(..., \mu_0,u_{i+j+1},...),$$
where $T_1$ and $T_2$ are the trees obtained from cutting $T$ along the edge $e$ and $j$ is the number of leaves in $T_2$.
	These last four identities prove that Equation (\ref{eq:pert_sum}) is equivalent to the $A_\infty$-algebra equation for the $\mu_k$. 
	
	The construction of map $\varphi: V \ra A$ is similar. We put $\varphi_{1}=i$ and $\varphi_{k}=\sum_{T\in\Gamma_k}\varphi_T$ where map $\varphi_T$ is defined in the same way as $\mu_T$, the only difference is that we assign $H$ to the root vertex (instead of $p$ as in the case of $\mu_T$). Similarly we define the auxiliary maps $$\hat\varphi_k=\sum_{\substack{T\in\Gamma_k\\e\in e(\overline T)}}(-1)^{\vert T_e\vert}\hat\varphi_{\overline T,e},$$ and for each $e\in e(T)$ we define $\varphi^{\Pi}_{T,e}$, $\varphi^{id}_{T,e}$ and $\varphi^{\gamma}_{T,e}$.
	
	The same argument we used above applies to show
	%$$\hat\varphi_k=-\sum_{\substack{T\in\Gamma_k(\beta)\\e\in e(T)}}(-1)^{\vert\overline T_e\vert}\varphi^{id}_{T,e}$$
	%and also
	%$$\hat\varphi_k=\sum_{\substack{T\in\Gamma_k(\beta)\\e\in e(T)}}(-1)^{\vert\overline T_e\vert}\big(\varphi^{\pi}_{T,e}-\varphi^{id}_{T,e}\big)+\sum_{\substack{T\in\Gamma_k(\beta)\\e\in E(T)\setminus e(T)}}(-1)^{\vert\overline T_e\vert}\hat\varphi_{T,e}.$$
	%Combining the two, we obtain
	\begin{align}\label{pert_mor}
	\sum_{T\in\Gamma_k}\Big(\sum_{e \in e(T)}(-1)^{\vert T_e\vert}&(\varphi^\Pi_{T,e}-\varphi^{\gamma}_{T,e})+\hat\varphi_{{T},r} +  \sum_{i=1}^{k} (-1)^{\vert T_e\vert}\hat\varphi_{{T},l_i}\Big)\nonumber\\
	& + \sum_{T\in\Gamma_k}\sum_{\substack{S\in\Gamma_{k+1}\\S/i=T}}(-1)^{\vert S_e\vert}\varphi_{S}(...u_{i-1}, \mu_0, u_{i}...)=0,
	\end{align}
	Now using (\ref{homotopy}) again we see that
	$$\hat\varphi_{T,r}=-m_1\circ\varphi_T+i\circ\mu_T-\varphi^{id}_{T,r}-\varphi^{\gamma}_{T,r},$$
	and 
	$$\varphi^{id}_{T,r}= m_{j}(\varphi_{T_1},\ldots, \varphi_{T_j}),$$ 
	where $j$ is the valency of the vertex of $T$ closest to the root and $T_i$ are the trees obtained from cutting $T$ at the incoming edges at that vertex.
    One can now see by the same argument that  Equation (\ref{pert_mor}) is equivalent to the $A_\infty$ homomorphism equation for $\varphi$:
	\begin{align}
		\sum_{j, \ i_1+\ldots+i_j=k}m_{j}\big(&\varphi_{i_1}(u_1,\ldots,u_{i_1}),\ldots,\varphi_{i_j}(\ldots,u_k)\big)\nonumber\\
		-&\sum_{\substack{0\leq j\leq k\\0\leq i\leq k-j}}(-1)^{*}\varphi_{k-j+1}\big(u_1,\ldots,\mu_{j}(u_{i+1},\ldots,u_{i+j}),\ldots,u_k\big)=0.\nonumber
	\end{align}
\end{proof}

\begin{rmk}
	In the case of uncurved $A_\infty$-algebras, there are also explicit formulas for a homomorphism $\psi:(A,m_k)\ra (V',\mu_k)$ with $\psi_1=p$ and a homotopy $\mathcal{H}:(A,m_k)\ra(A,m_k)$ between $\varphi\circ \psi$ and $id_A$. See \cite{markl} for this construction.
\end{rmk}

\subsection{The $L_\infty$ case} The discussion in the $L_\infty$ case is very much the same as the $A_\infty$ case, except that instead of using planar stable rooted trees in the formulas, one uses isomorphism classes of stable rooted trees. 

The only difference is how to define the map $\mu_T$ for each tree $T$ (as opposed to a planar tree): we pick $\widetilde{T}$ a planar embedding of $T$ and define $\mu_{\widetilde{T}}$ as before. Then we take $\mu_T= \frac{1}{|{\sf Aut}( T)|}\mu_{\widetilde{T}}\circ {\sf Sh}$ where ${\sf Sh}$ is the symmetrization map and $|{\sf Aut}(T)|$ is the order of the automorphism group of $T$. We refer to~\cite[Section 4]{FM} for a detailed treatment of this construction.

The rest of the proof is exactly the same.

\section{The category of curved $L_\infty$ spaces}\label{sec:spaces}

In this section, we recall basic definitions of curved $L_\infty$ spaces, morphisms between these spaces and describe the notion of homotopy between morphisms.

\subsection{Curved $L_\infty$ spaces} A curved $L_\infty$ space (sometimes shortened to $L_\infty$ space) is a pair $(M,\mathfrak{g})$ where $M$ is a smooth manifold, and $\mathfrak{g}$ is a $\Z$-graded vector bundle over $M$ of the form
\[ \mathfrak{g}=\mathfrak{g}_2\oplus \mathfrak{g}_3 \oplus \cdots\oplus \mathfrak{g}_d \]
for some $d\geq 2$, together with bundle maps $\mu_k: {\sf sym}^k (\mathfrak{g}[1]) \ra \mathfrak{g}[1]$ of degree one such that the $L_\infty$ equation holds: 
\[ \sum_{k=0}^n \sum_{\sigma\in {\sf Sh}(k,n-k)} \epsilon_\sigma\cdot \mu_{n-k+1}\big(\mu_k(a_{\sigma(1)},\ldots,a_{\sigma(k)}),a_{\sigma(k+1)},\ldots,a_{\sigma(n)}\big)=0\]
where ${\sf Sh}(k,n-k)$ consists of $(k,n-k)$-type shuffles, and $\epsilon_\sigma$ is the Koszul sign associated with the permutation $a_1\otimes\cdots\otimes a_n \mapsto a_{\sigma(1)},\ldots,a_{\sigma(n)}$ with the $a$'s considered as elements of $\mathfrak{g}[1]$.

In order to formulate a good notion of homotopy between morphisms of $L_\infty$ spaces we will need ``special" connections on $T_M$ and $\mathfrak{g}$. Therefore we make the following assumption:
$$ \textrm{$T_M$ has a torsion-free, flat connection and $\mathfrak{g}$ has a flat connection.}$$
 In fact, it would be enough for most purposes to require the existence of these connections on an open neighborhood of $\mu_0^{-1}(0)$. But for simplicity we stick to the whole $M$.

 The main results in this paper are local, meaning $M$ is an open ball in $\mathbb{R}^n$, therefore this assumption is trivially satisfied.
\medskip

A morphism from $(M,\mathfrak{g})$ to $(N,\mathfrak{h})$ is a pair $\mathfrak{f}=(f,f^\sharp)$ where $f: M\ra N$ is a smooth map, and $f^\sharp: \mathfrak{g} \ra f^*\mathfrak{h}$ is a homomorphism of $L_\infty$ algebras. This means a sequence of (degree zero) bundle maps $f^\sharp_k:  {\sf sym}^k(\mathfrak{g}[1]) \ra f^*\mathfrak{h}[1]$ satisfying
\begin{align*}
 \sum_k \frac{1}{k!}\sum_{\sigma}&\epsilon_\sigma\cdot \mu_k\big(f^\sharp_{i_1}(a_{\sigma(1)}\cdots)\cdots f^\sharp_{i_k}(\cdots a_{\sigma(n)})\big)\\
 = & \sum_{r}\sum_\tau \epsilon_\tau\cdot f^\sharp_{n-r+1}\big(\mu_r(a_{\tau(1)}\cdots a_{\tau(r)})\cdots a_{\tau(n)}\big),\end{align*}
where $\sigma$ is a $(i_1,\cdots,i_k)$ type shuffle, and $\tau$ is a $(r,n-r)$ type shuffle. On the left-hand side there is an abuse of notation: $\mu_k$ stands for $f^*\mu_k$.

Morphisms of $L_\infty$ spaces can be composed similarly to the algebra case. Given $L_\infty$ morphisms $\mathfrak{e}: (M',\mathfrak{g}') \to (M,\mathfrak{g}) $ and $\mathfrak{f}: (M,\mathfrak{g}) \to (N,\mathfrak{h}) $ we define $\mathfrak{f}\circ\mathfrak{e}:=(f\circ e, f^\sharp\circ e^\sharp)$ where
\begin{equation}\label{eq:comp}
(f^\sharp\circ e^\sharp)_n(a_1\cdots a_k) =\frac{1}{k!}\sum_{\sigma}\epsilon_\sigma\cdot e^*(f^\sharp_k)\big(e^\sharp_{i_1}(a_{\sigma(1)}\cdots)\cdots e^\sharp_{i_k}(\cdots a_{\sigma(n)})\big)
\end{equation}

As in the algebra case, we also define  $L_{(n)}$ morphisms between curved $L_\infty$ spaces.
%In order to formulate a good notion of homotopy between morphisms of $L_\infty$ spaces, it is necessary to involve the tangent bundle of $M$ as part of the $L_\infty$ structure.

\subsection{Extensions of $L_\infty$ structures} Let $(M,\mathfrak{g})$ be a $L_\infty$ space. By our assumptions, we can choose a torsion-free, flat connection on $T_M$ and also a flat connection on the bundle $\mathfrak{g}$. We set
\[ \widetilde{\mathfrak{g}}:= T_M \oplus \mathfrak{g},\]
with $T_M$ at cohomological degree one. The $L_\infty$ structure on $\mathfrak{g}$ naturally extends to $\widetilde{\mathfrak{g}}$ by inductively applying the following formula
\begin{equation}~\label{eq:extension}
  \mu_{k+1}(X\cdot \alpha_1\cdot\cdots\cdot\alpha_k):= \nabla_X\mu_k(\alpha_1\cdot\cdots\cdot\alpha_k)- \sum_{j=1}^k \mu_k(\alpha_1\cdot\cdots\nabla_X\alpha_j\cdots\cdot\alpha_k)
\end{equation}
Using the torsion freeness and the flatness, one can verify that when pulling out tangent vectors using the above formula, the choice of order does not matter, i.e. we have that
\[ \mu_{k+2}(X\cdot Y \cdot \alpha_1\cdots\alpha_k ) = \mu_{k+2}(Y\cdot X \cdot \alpha_1\cdots\alpha_k)\]
for any two tangent vectors $X,Y\in T_M$. 

\begin{lem}
 Equation~\eqref{eq:extension} defines a $L_\infty$ algebra structure on $\widetilde{\mathfrak{g}}$.
\end{lem}

\begin{proof}
We prove the $L_\infty$ identity by induction on the total number of tangent vectors. Indeed, when there is no tangent vector, the $L_\infty$ identity holds since $\mathfrak{g}$ forms an $L_\infty$ algebra to begin with. We want to verify the $L_\infty$ identity:
\[ \sum_{k=1}^n \sum_{\sigma\in {\sf Sh}(k,n-k)} \epsilon_\sigma\cdot \mu_{n-k+1}\big(\mu_k(a_{\sigma(1)},\ldots,a_{\sigma(k)}),a_{\sigma(k+1)},\ldots,a_{\sigma(n)}\big)=0\]
It is enough to consider the case when all the inputs $a$'s are flat with respect to the chosen connection $\nabla$. Now we pick a tangent vector, say $a_1$, among the inputs and apply Equation~\eqref{eq:extension} to pull it out of the inputs. If $a_1$ falls into $a_{\sigma(1)},\ldots,a_{\sigma(k)}$, we obtain terms of the form
\[ \sum\sum \epsilon_\sigma \mu_{n-k+1}\big(\nabla_{a_1}\mu_{k-1}(\cdots),\cdots\big)\]
When $a_1$ falls into $a_{\sigma(k+1)},\ldots,a_{\sigma(n)}$, we get
\[\nabla_{a_1}\big( \sum\sum \epsilon_\sigma \mu_{n-k}(\mu_k(\cdots),\cdots)\big)-\sum\sum \epsilon_\sigma \mu_{n-k+1}\big(\nabla_{a_1}\mu_{k-1}(\cdots),\cdots\big)\]
Thus, their sum yields $\nabla_{a_1}\big( \sum\sum \epsilon_\sigma \mu_{n-k}(\mu_k(\cdots),\cdots)\big)$ which vanishes by induction.
\end{proof}

$L_\infty$ morphisms between $L_\infty$ spaces can also be extended to the tangent bundles. More precisely, let $(f,f^\sharp): (M,\mathfrak{g}) \ra (N, \mathfrak{h})$ be a morphism of $L_\infty$ spaces
%In particular, we have an $L_\infty$ morphism \[ f^\sharp: \mathfrak{g}\ra f^*\mathfrak{h}\]
and, as above, choose torsion free and flat connections on both spaces and consider the extended $L_\infty$ algebras $\widetilde{\mathfrak{g}}, \widetilde{\mathfrak{h}}$. We extend the homomorphism $f^\sharp$ to a homomorphism
\[ f^\sharp: \widetilde{\mathfrak{g}} \ra f^*\widetilde{\mathfrak{h}},\]
which we still denote by $f^\sharp$.
The formula of extension is the same as in Equation~\eqref{eq:extension}, i.e. we inductively define
\begin{equation}~\label{eq:ext-mor}
 f^\sharp_{k+1} (X\cdot \alpha_1\cdots\alpha_k): = \nabla_Xf^\sharp_k(\alpha_1\cdot\cdots\cdot\alpha_k)- \sum_{j=1}^k f^\sharp_k(\alpha_1\cdot\cdots\nabla_X\alpha_j\cdots\cdot\alpha_k)
\end{equation}
The difference is that here we need $k\geq 1$. When $k=0$ we define the map $f_1^\sharp: T_M \ra f^*T_N$ to be the tangent map $df$.

\begin{lem}
The maps defined in Equation~\eqref{eq:ext-mor} form a $L_\infty$ morphism $f^\sharp: \widetilde{\mathfrak{g}} \ra f^*\widetilde{\mathfrak{h}}$.
\end{lem}
\begin{proof}
We need to verify that
\begin{align*}
 \sum_k \frac{1}{k!}\sum_{\sigma}\epsilon_\sigma\cdot& \mu_k\big(f^\sharp_{i_1}(a_{\sigma(1)}\cdots)\cdots f^\sharp_{i_k}(\cdots a_{\sigma(n)})\big)\\
= & \sum_{r}\sum_\tau \epsilon_\tau\cdot f^\sharp_{n-r+1}\big(\mu_r(a_{\tau(1)}\cdots a_{\tau(r)})\cdots a_{\tau(n)}\big)\end{align*}
 Let us pick up a tangent vector, say $a_1$ among the inputs. Also we assume that all the input vectors are flat. If $a_1$ is inside $f_{i_j}$, and $i_j=1$, we get 
\begin{align*}
&  \sum_k \frac{1}{(k-1)!}\sum_{\sigma}\epsilon_\sigma\cdot \nabla_{a_1} \mu_{k-1}\big(f^\sharp_{i_1}(a_{\sigma(1)}\cdots)\cdots \widehat{f_{1}(a_1)}\cdots f^\sharp_{i_{k}}(\cdots a_{\sigma(n)})\big)\\
-&  \sum_k \frac{1}{(k-1)!}\sum_{\sigma}\epsilon_\sigma\cdot  \mu_{k-1}\big(f^\sharp_{i_1}(a_{\sigma(1)}\cdots)\cdots \widehat{f_{1}(a_1)}\cdots\nabla_{a_1}f^\sharp_{i_l}(\cdots)\cdots f^\sharp_{i_{k}}(\cdots a_{\sigma(n)})\big)
\end{align*}
The coefficient becomes $\frac{1}{(k-1)!}$ since there are $k$ possible choices of $j$.
The second term cancels precisely the terms with $a_1$ inside $f_{i_j}$ with $i_j\geq 2$. Thus the left hand side is equal to (by induction on the total number of tangent vectors)
\begin{align*}
& \nabla_{a_1}\Big( \sum_k \frac{1}{(k-1)!}\sum_{\sigma}\epsilon_\sigma\cdot  \mu_{k-1}\big(f^\sharp_{i_1}(a_{\sigma(1)}\cdots)\cdots \widehat{f_{1}(a_1)}\cdots f^\sharp_{i_{k}}(\cdots a_{\sigma(n)})\Big)\\
 =&  \nabla_{a_1}\Big( \sum_{r}\sum_\tau \epsilon_\tau\cdot f^\sharp_{n-r}\big(\mu_r(\cdots)\cdots \big)\Big)\\
 = & \sum_{r}\sum_\tau \epsilon_\tau\cdot f^\sharp_{n-r+1}\big(\mu_r(a_{\tau(1)}\cdots a_{\tau(r)})\cdots a_{\tau(n)}\big)
 \end{align*}
 which is exactly the right hand side.
\end{proof}

These extensions of $L_\infty$ spaces induced by the choice of connections are in fact independent of these choices up to isomorphism.

\begin{lem}\label{lem:independence}
	Let $\nabla$ and $\nabla'$ be torsion free and flat connections. Let $\widetilde{\mathfrak{g}}$ and $\widetilde{\mathfrak{g}}'$ be the associated extended $L_\infty$ algebras. Then there is an isomorphism
	\[ \Phi^\mathfrak{g}: \widetilde{\mathfrak{g}} \to \widetilde{\mathfrak{g}}'\]
	defined by $\Phi_1^\mathfrak{g}=\id$, $\Phi_2^\mathfrak{g}(X,\alpha)=(\nabla'_X - \nabla_X)(\alpha)$, and for $k\geq 3$ by the recursive formula
	\[ \Phi_{k}^\mathfrak{g}(X\cdot \alpha_1\cdot\cdots\cdot\alpha_{k-1}):= \nabla_X'\Phi_{k-1}^\mathfrak{g}(\alpha_1\cdot\cdots\cdot\alpha_{k-1})- \sum_{j=1}^{k-1} \Phi_{k-1}^\mathfrak{g}(\alpha_1\cdot\cdots\nabla_X\alpha_j\cdots\cdot\alpha_k).\]
	
	Moreover this isomorphism is natural:  given a $L_\infty$ morphism $(f,f^\sharp): (M,\mathfrak{g}) \to (N, \mathfrak{h})$ and different choices of connections, the induced extended $L_\infty$ morphisms $ f^\sharp: \widetilde{\mathfrak{g}} \ra f^*\widetilde{\mathfrak{h}}$ and $ (f^\sharp)': \widetilde{\mathfrak{g}'} \ra f^*\widetilde{\mathfrak{h}'}$ satisfy $\displaystyle \Phi^\mathfrak{h} \circ f^\sharp = (f^\sharp)' \circ \Phi^\mathfrak{g}$.
	
\end{lem}
\begin{proof}
	As before we prove the $L_\infty$ homomorphism equation by induction on the number of inputs that are tangent vectors. When there is no tangent vector, the operations $\mu_k$ and $\mu'_k$ agree and $\Phi$ is just the identity. Let us now pick up a tangent vector, say $a_1$ among the inputs. For simplicity we assume that all the inputs are flat with respect to $\nabla$. Let's consider the left-hand side of the $L_\infty$ equation, when $a_1$ is inside $\Phi_{i_j}$ with $i_j=1$, we get 
	\begin{align*}
	&  \sum_k \frac{1}{(k-1)!}\sum_{\sigma}\epsilon_\sigma\cdot \nabla'_{a_1} \mu'_{k-1}\big(\Phi_{i_1}(a_{\sigma(1)}\cdots)\cdots \widehat{\Phi_{1}(a_1)}\cdots \Phi_{i_{k}}(\cdots a_{\sigma(n)})\big)\\
	-&  \sum_k \frac{1}{(k-1)!}\sum_{\sigma}\epsilon_\sigma\cdot  \mu'_{k-1}\big(\Phi_{i_1}(a_{\sigma(1)}\cdots)\cdots \widehat{\Phi_{1}(a_1)}\cdots\nabla'_{a_1}\Phi_{i_l}(\cdots)\cdots \Phi_{i_{k}}(\cdots a_{\sigma(n)})\big),
	\end{align*}
	using the definition of $\mu'_k$ and the fact $\Phi_1=\id$. The second sum above exactly cancels with the other terms in the left-hand side of the $L_\infty$-homomorphism equation with $a_1$ inside $f_{i_j}$ with $i_j\geq 2$. This is because the $a_i$ are $\nabla$ flat and $\Phi_2(a_1, a_j)= \nabla'_{a_1}a_j$. Therefore, by induction hypothesis, the left-hand side equals
	\begin{align}\label{eq:phi}
	\nabla'_{a_1}\Big( \sum_{r}\sum_\tau \epsilon_\tau\cdot \Phi_{n-r}\big(&\mu_r(a_{\tau(2)}\cdots)\cdots a_{\tau(n)}\big)\Big)  =   \nabla'_{a_1}\big(\mu_{n-1}(a_{2}\cdots a_{n})\big)+\nonumber\\
	& + \sum_{r\leq n-2}\sum_\tau \epsilon_\tau\cdot \Phi_{n+1-r}\big(a_1 \cdot \mu_r(a_{\tau(2)}\cdots)\cdots a_{\tau(n)}\big)\\
	& + \sum_{r\leq n-2}\sum_\tau \epsilon_\tau\cdot \Phi_{n-r}\big(\nabla_{a_1}\big(\mu_r(a_{\tau(2)}\cdots)\big)\cdots a_{\tau(n)}\big).\nonumber
	\end{align}
	Here the first term equals $\Phi_2(a_1, \mu_{n-1}(a_{2}\cdots a_{n}))$ and in the third term we have $\nabla_{a_1}\big(\mu_r(a_{\tau(2)}\cdots)=\mu_{r+1}(a_1 \cdot a_{\tau(2)}\cdots)$. Hence (\ref{eq:phi}) equals
	\[\sum_{r}\sum_\tau \epsilon_\tau\cdot \Phi_{n-r+1}\big(\mu_r(a_{\tau(1)}\cdots a_{\tau(r)})\cdots a_{\tau(n)}\big)\]
	More precisely, the first two terms in (\ref{eq:phi}) correspond above to the terms where $a_1$ is outside the $\mu_r$. 
	
	The proof of the naturality statement is entirely analogous and we omit it.
\end{proof}

\subsection{Homotopy of $L_\infty$ morphisms}
In order to define the notion of homotopy we need to consider the space version of tensoring with $\Omega_{[0,1]}^*$ as in Subsection \ref{subsec:Ahomotopy}. Let $(M,\mathfrak{g})$ and $(N,\mathfrak{h})$ be two $L_\infty$ spaces and let $F: M\times [0,1] \ra N$ be a smooth map. Consider the graded bundle $F^*\widetilde{\mathfrak{h}}_{[0,1]}:=F^*\widetilde{\mathfrak{h}}\otimes \pi_2^*\Omega_{[0,1]}^*$, where $\pi_2: M\times [0,1] \ra [0,1]$ is the projection and denote $\mu^t_k= (f^t)^*\mu_k$, $f^t=F(-,t)$. On $F^*\widetilde{\mathfrak{h}}_{[0,1]}$ we define the operations
\begin{align}
\mu^{\otimes}_0 & := \mu^t_0 - (dF/dt)dt\nonumber\\
\mu^{\otimes}_1(x(t)+y(t)dt) & := \mu^t_1(x(t))+\mu^t_1(y(t))\,dt + (-1)^{|x(t)|} \nabla_{d/dt}(x(t)))dt\\
\mu^{\otimes}_k(...,x_i(t)+y_i(t)dt,... ) & := \mu^t_k(..., x_i(t),...) + \sum_i (-1)^\dagger \mu^t_k(x_1(t),...,y_i(t),...x_k(t))dt,  \nonumber
\end{align}
for $ k\geq 2$, where $\dagger=\sum_{a=i+1}^{k} |x_a|'$.

\begin{lem}
	The operations $\mu^\otimes_k$ define a curved $L_\infty$ algebra structure on $F^*\widetilde{\mathfrak{h}}_{[0,1]}$.
	%Moreover there are naive $L_\infty$ morphisms $\ev_i: F^*\widetilde{\mathfrak{h}}_{[0,1]} \ra \iota_i^* \mathfrak{h} $, $i=0,1$ given by $\ev_i(x(t)+y(t)dt)= x(i)$.
\end{lem}
\begin{proof}
The proof is standard, it follows from the fact that the tensor product of a $L_\infty$-algebra and a commutative dg-algebra is again a $L_\infty$-algebra together with the relation
$$\nabla_{\frac{d}{dt}} \mu^t_k(a_1,\ldots, a_k)= \mu_{k+1}^\otimes(\frac{\partial F}{\partial t}dt, a_1, \ldots, a_k),$$
for flat $a_i$.	
\end{proof}

We are now ready to define homotopy.

\begin{defi}\label{def:homotopy}
Two $L_\infty$ morphisms $(f^0, f^{0,\sharp})$ and $(f^1,f^{1,\sharp})$ from $(M,\mathfrak{g})$ to $(N,\mathfrak{h})$ are homotopic if there exists a map $F: M\times [0,1] \ra N$, together with an $L_\infty$ homomorphism
\[ F^\sharp:\pi_1^*\widetilde{\mathfrak{g}}\ra F^*\widetilde{\mathfrak{h}}_{[0,1]},\]
where $\pi_1: M\times [0,1] \ra M$ is the projection map, satisfying the following conditions:
\begin{itemize}
\item It's compatible with the connection, i.e. 
$$F^\sharp_{k+1} (X\cdot \alpha_1\cdots\alpha_k): = \nabla_XF^\sharp_k(\alpha_1\cdot\cdots\cdot\alpha_k)- \sum_{j=1}^k F^\sharp_k(\alpha_1\cdot\cdots\nabla_X\alpha_j\cdots\cdot\alpha_k)$$
and $F_1^\sharp(X)=dF(X), \; \forall X\in T_M$.
\item The following boundary conditions hold:
\begin{align*}
 (F,F^\sharp)|_{t=0} & =(f^0,f^{0,\sharp})\\
 (F,F^\sharp)|_{t=1} & =(f^1,f^{1,\sharp})
 \end{align*}
\end{itemize}
Note that by Lemma~\ref{lem:independence}, this definition is independent of the choice of $\nabla$. 
\end{defi}

Like usual we say a $L_\infty$ morphism $\mathfrak{f}: \mathbb{M}\to\mathbb{N}$ is a homotopy equivalence if there is another $L_\infty$ morphism $\mathfrak{e}: \mathbb{N}\to\mathbb{M}$ such that both $\mathfrak{f}\circ \mathfrak{e}$ and $\mathfrak{e}\circ \mathfrak{f}$ are homotopic to the identity $L_\infty$ morphism.

\begin{rmk}\label{rmk:analytic_homotopy}
The above definition can be easily adapted to the complex analytic setting when both $(M,\mathfrak{g})$ and $(N,\mathfrak{h})$ are holomorphic $L_\infty$ spaces such that the underlying complex manifolds $M$ and $N$ admits holomorphic torsion free and flat connections. More precisely, we simply require that $F$ and $F^\sharp$ be fiberwise holomorphic, i.e. they are smooth in the $t$ direction, and holomorphic whenever we fix a value $t\in [0,1]$. This is more transparent with explicit formulas of $F^\sharp$ in the following paragraph.
\end{rmk}

It will be helpful to unwind this definition. The compatibility condition implies the morphism $F^\sharp$ is determined by its values on the elements of $\pi_1^*{\mathfrak{g}}$. We write 
$$F^\sharp_k(a_1,\ldots, a_k)= f^t_k(a_1,\ldots, a_k)+ (-1)^{\sum_i |a_i|'}h^t_k(a_1,\ldots, a_k)dt.$$
 Then the $L_\infty$ morphism equation for $F^\sharp$ is equivalent to
\begin{enumerate}
 \item The maps $(f^t_k)_{k\geq 1}$ define an $L_\infty$ homomorphism $\mathfrak{g} \ra (f^t)^*\mathfrak{h}$;
 \item The maps $h^t_k$ satisfy the equations $h_1^t(\mu_0)= \frac{\partial F}{\partial t}$ and for $n\geq 1$
 \begin{align}\label{eq:homotopy} 
 \sum_k &\frac{1}{(k-1)!}\sum_{\sigma}\epsilon_\sigma  \mu^t_k\big(h^t_{i_1}(a_{\sigma(1)}\cdots) f^t_{i_2}\cdots f^t_{i_k}(\cdots a_{\sigma(n)})\big)+\nonumber\\
  & \sum_{j\geq 0}\sum_\tau \epsilon_\tau  h^t_{n-j+1}\big(\mu_j(a_{\tau(1)}\cdots a_{\tau(r)})\cdots a_{\tau(n)}\big) = \nabla_{\frac{d}{dt}} f^t_n(a_1,\ldots,a_n)
  \end{align}
 where $\sigma$ is a $(i_1,\cdots,i_k)$ type shuffle, $\tau$ is a $(r,n-r)$ type shuffle and the $a_i$ are flat.
\end{enumerate}

\begin{prop}\hfill
	
	\begin{itemize}
		 \item[\textbf{(a)}] Homotopy of $L_\infty$ morphisms is an equivalence relation.
		 \item[\textbf{(b)}] Let $(f^0, f^{0,\sharp})$ and $(f^1,f^{1,\sharp})$ be homotopic $L_\infty$ morphisms. Then $(f^0, f^{0,\sharp})\circ (d, d^\sharp), \  (f^1, f^{1,\sharp})\circ (d, d^\sharp)$ are homotopic and $(e, e^\sharp) \circ (f^0, f^{0,\sharp}), \  (e, e^\sharp) \circ (f^1,f^{1,\sharp})$ are homotopic, for any composable $L_\infty$ morphisms $(d, d^\sharp), (e, e^\sharp)$.
	\end{itemize}
\end{prop}	
\begin{proof}
 For \textbf{(a)} first note that a diffeomorphism $\rho:[0,1]\to [0,1]$ induces, by pull-back, a $L_\infty$ homomorphism $ \rho^*:  F^*\widetilde{\mathfrak{h}}_{[0,1]} \to (F_\rho)^*\widetilde{\mathfrak{h}}_{[0,1]}$, where $F_\rho:=F\circ(\id\times \rho)$. Now given a homotopy $(F, F^\sharp)$ from $(f^0, f^{0,\sharp})$ to $(f^1,f^{1,\sharp})$, take $\rho(t)=1-t$ and consider the pair $(F_\rho,  F^\sharp_\rho:=\rho^*\circ F^\sharp)$. This defines a homotopy from $(f^1,f^{1,\sharp})$ to $(f^0, f^{0,\sharp})$, which shows symmetry of the homotopy relation. For transitivity let $\rho$ be a non-decreasing diffeomorphism which is constant in neighborhoods of $0$ and $1$ in $[0,1]$. For this choice of $\rho$, $\mathfrak{F}_\rho:=(F_\rho,  F^\sharp_\rho)$ is a new homotopy  from $(f^0, f^{0,\sharp})$ to $(f^1,f^{1,\sharp})$. Given a homotopy $(G, G^\sharp)$ from $(f^1,f^{1,\sharp})$ to $(f^2,f^{2,\sharp})$ we consider $\mathfrak{G}_\rho$, as before and define the concatenation $\mathfrak{F}_\rho \bullet \mathfrak{G}_\rho$ by
 	\[  F_\rho \bullet G_\rho (x, t)= \left\{ {\begin{array}{ll}
 	F_\rho (x, 2t), & t\leq 1/2 \\
 	G_\rho (x,2t-1), & t \geq 1/2\\
 	\end{array} } \right.
 \]
 and analogously $F_\rho^\sharp \bullet G_\rho^\sharp$. By our choice of $\rho$ these are smooth maps and can be easily seen to determine a homotopy from $(f^0, f^{0,\sharp})$ to $(f^2, f^{2,\sharp})$.
 
 For \textbf{(b)} we prove only the second statement as they are analogous. Let $(F, F^\sharp)$ be a homotopy from $(f^0, f^{0,\sharp})$ to $(f^1,f^{1,\sharp})$ and $e^\sharp : \mathfrak{h}\to e^*\mathfrak{h}'$ be a $L_\infty$ homomorphism. It is easy to check there is an induced $L_\infty$ homomorphism $\widetilde{e}^\sharp : F^*\widetilde{\mathfrak{h}}_{[0,1]} \to (e\circ F)^*\widetilde{\mathfrak{h}'}_{[0,1]}$. Now the pair $(e\circ F, \widetilde{e}^\sharp\circ F^\sharp)$ defines the required homotopy from  $(e, e^\sharp) \circ (f^0, f^{0,\sharp})$ to  $(e, e^\sharp) \circ (f^1,f^{1,\sharp})$.
\end{proof}

\section{The inverse function theorem for $L_\infty$ spaces}~\label{sec:theorems}

In this section, we first adapt the obstruction theory of Section~\ref{sec:obstruction} to the case of $L_\infty$ spaces. Then we prove Theorem~\ref{thm:main} and Theorem~\ref{thm:main2}.

\subsection{ Obstruction theory for morphisms between $L_\infty$ spaces} Much of the discussion on the obstruction theory for $A_\infty$ and $L_\infty$ homomorphisms in Section 2 translates without significant changes to the $L_\infty$ space setting. When we are given two $L_\infty$ spaces $(M,\mathfrak{g})$, $(N, \mathfrak{h})$ and a smooth map $f: M \ra N$, we can define a differential $\delta$ on $\oplus_k \Hom({\sf sym}^k(\mathfrak{g}[1]), f^*\mathfrak{h}[1])$, as in Subsection~\ref{subsec:obs}, 
$$\delta (\phi_k) (a_1,\ldots,a_{k-1}) := (-1)^{|\phi_k|'}\phi_k(\mu_0,a_1,\ldots,a_{k-1}).$$
We denote by $D^k(\mathfrak{g},f^*\mathfrak{h})$ the $\delta$ cohomology and, assuming there is a map $f^\sharp_1$ satisfying $f^\sharp_1(\mu_0)=f^*\mu_0$, we define the differential 
$$d\phi(a_1,\ldots, a_k):= \mu_1\phi(a_1,\ldots, a_k)-(-1)^{|\phi|'+|a_i|'\sum_{l=1}^{i-1}|a_l|'}\phi(\mu_1(a_i),a_1\ldots,\widehat{a_i},\ldots, a_k),$$
as before $\mu_1$ in the first term really stands for $f^*\mu_1$.
As in Definition \ref{defi:obs}, we define the obstruction space $$H^k(\mathfrak{g},f^*\mathfrak{h}) := H^1\big( D^k(\mathfrak{g},f^*\mathfrak{h}), d\big).$$

We also define a sequence of maps $(f^\sharp_1,\ldots,f^\sharp_n): \mathfrak{g} \to f^*\mathfrak{h}$ (together with $f$) to be an $L_{(n)}$ morphism, if it satisfies the $L_\infty$ homomorphism equation for $0\leq k\leq n-1$ inputs and for $n$ inputs up to a $\delta$-exact  term. We define, in the same way as in Section~\ref{sec:obstruction}, an obstruction class
\[ \mathfrak{o}\big( (f^\sharp_j)_{j=1}^n \big)\in H^{n+1}(\mathfrak{g},f^*\mathfrak{h}).\]
This obstruction class vanishes if and only if the map can be lifted to a $L_{(n+1)}$ morphism.

\begin{lem}
The obstruction class $\mathfrak{o}\big( (f^\sharp_j)_{j=1}^n\big)\in H^{n+1}(\mathfrak{g},f^*\mathfrak{h}) $ vanishes if and only if the corresponding class (in the extended $L_\infty$ algebras) $\mathfrak{o}\big( (f^\sharp_j)_{j=1}^n \big) \in H^{n+1}( \widetilde{\mathfrak{g}}, f^*\widetilde{\mathfrak{h}})$ vanishes.
\end{lem}
\begin{proof}
If there exists $f^\sharp_{n+1}: {\sf sym}^{n+1}(\mathfrak{g}[1])\ra f^*\mathfrak{h}[1]$, we may extend it using Equation~\eqref{eq:ext-mor} to obtain an $L_{(n+1)}$ homomorphism on the extended $L_\infty$ algebras, which implies that $\mathfrak{o}\big( (f^\sharp_j)_{j=1}^n \big) \in H^{n+1}( \widetilde{\mathfrak{g}}, f^*\widetilde{\mathfrak{h}})$ vanish. Conversely, if the latter obstruction class vanishes, we simply restrict the map $f^\sharp_{n+1}$ to ${\sf sym}^{n+1}(\mathfrak{g}[1])\ra f^*\mathfrak{h}[1]$.
\end{proof}

Extra work is needed to formulate the homotopy invariance of obstruction spaces and classes. Let $(F, F^\sharp)$ be a homotopy between two $L_{(1)}$ morphisms $(f^0,f^{0,\sharp})$ and $(f^1,f^{1,\sharp})$. Denote by $\iota_a: M \ra M \times [0,1]$ the inclusion map $\iota_a(x)=(x,a)$.  For an element $\varphi \in \Hom({\sf sym}^{n+1}\pi_1^*\widetilde{\mathfrak{g}}[1], F^*\widetilde{\mathfrak{h}}_{[0,1]}[1])$, we have $\iota_a^*(\varphi) \in \Hom({\sf sym}^{n+1}\widetilde{\mathfrak{g}}[1], (f^a)^*\widetilde{\mathfrak{h}})$, for $a=0,1$. 
It is easy to see this assignment induces a map on obstruction spaces $\ev_a: H^{n+1}(\pi_1^*\widetilde{\mathfrak{g}}, F^*\widetilde{\mathfrak{h}}_{[0,1]}) \ra H^{n+1}(\widetilde{\mathfrak{g}}, (f^a)^*\widetilde{\mathfrak{h}})$, which we call the evaluation map.

\begin{prop}\label{prop:ev_iso}
The evaluation map
\[ {\sf ev}_a: H^{n+1}(\pi_1^*\widetilde{\mathfrak{g}}, F^*\widetilde{\mathfrak{h}}_{[0,1]}) \ra H^{n+1}(\widetilde{\mathfrak{g}}, (f^a)^*\widetilde{\mathfrak{h}})\]
is an isomorphism.
\end{prop}
\begin{proof}
Both cases are identical, we will prove the statement for $a=0$, by constructing $\sf{i}$ a homotopy inverse to $\ev_0$. 
Given $\varphi \in \Hom({\sf sym}^{n+1}\widetilde{\mathfrak{g}}[1], (f^0)^*\widetilde{\mathfrak{h}})$, we define $\sf{i}\varphi (a_1,\ldots, a_k)$ by taking the (unique) flat extension of $\varphi (a_1,\ldots, a_k)$ in the $t$-direction  and so obtain an element of $F^*\widetilde{\mathfrak{h}}$. It is clear $\sf{i}$ commutes with the $\delta$ differential and hence it induces a map from $D^k(\widetilde{\mathfrak{g}}, (f^0)^*\widetilde{\mathfrak{h}})$ to  $D^k(\pi_1^*\widetilde{\mathfrak{g}},F^*\widetilde{\mathfrak{h}}_{[0,1]})$. We pick a flat frame of the bundle $F^*\widetilde{\mathfrak{h}}$ and compute, for a $\delta$-closed $\varphi$,
 \begin{align*}
d(\textsf{i}(\varphi))- \textsf{i}( d(\varphi)) & = (\mu^t_1- \mu^0_1)\circ\textsf{i}\varphi = \int_0^t \nabla_{d/ds} \mu^s_1\circ\textsf{i}\varphi\, ds\\
& = \int_0^t \mu^s_2(\frac{\partial F}{\partial s}, \textsf{i}\varphi) \,ds = \int_0^t \mu^s_2(h^s_1(\mu_0), \textsf{i}\varphi) \,ds\\
& = \delta \big( \int_0^t \mu^s_2(h^{s}_1 \cdot \textsf{i}\varphi)\, ds\big),
\end{align*}
Here $h_1^t$ is the map coming from the definition of homotopy in (\ref{eq:homotopy}). Hence $\sf i$ induces a map between the corresponding obstruction spaces.
Clearly, we have $\ev_0\circ \textsf{i}=\id$. Let
\[K: D^k(\pi_1^*\widetilde{\mathfrak{g}},F^*\widetilde{\mathfrak{h}}_{[0,1]}) \to D^k(\pi_1^*\widetilde{\mathfrak{g}},F^*\widetilde{\mathfrak{h}}_{[0,1]})
\]
be the map induced by the integration map $$K(\varphi^t +\psi^t dt)(a_1,\ldots,a_k)= (-1)^{|\psi|'+\star}\int_0^t \psi^s(a_1,\ldots,a_k)\,ds,$$
where $\star:=|a_1|'+\ldots +|a_k|'$. We claim that $\textsf{i}\circ {\ev_0}-\id= d K+ K d$. For $\phi:= \varphi^t +\psi^t dt$ we compute (omitting the inputs)
\begin{align*}
(d K +K d)(\phi)  =  & \ (-1)^{|\psi|'+\star}\mu_1^t\int_0^t \psi^s\, ds- \nabla_{\frac{d}{dt}} \int_0^t \psi^s\,ds\,dt - (-1)^{\star+1}\int_0^t\psi^s \widetilde{\mu_1}\,ds\\
& - \int_0^t \nabla_{\frac{d}{ds}}\varphi^s\,ds + (-1)^{|\psi|'+\star+1}\int_0^t\mu_1^s \psi^s\,ds - (-1)^{\star}\int_0^t \psi^s \widetilde{\mu_1}\,ds\\
= & \ (-1)^{|\psi|'+\star}\left(\int_0^t \nabla_{\frac{d}{ds}}\mu_1^sK(\psi^s ds)\,dt-\int_0^t \mu_1^s(\nabla_{\frac{d}{ds}}K(\psi^s ds))\,dt\right)\\
 & \ -\psi^tdt - \varphi^t + \varphi^0 \\
= & \  \textsf{i}\ev_0(\phi) - \phi +(-1)^{|\psi|'+\star} \int_0^t \mu_2^s(\frac{ \partial F}{\partial s} \cdot K(\psi^s ds))\,dt\\
= & \ \textsf{i}\ev_0(\phi)-\phi  - (-1)^{|\psi|'+\star}\delta \big( \int_0^t \mu_2^s(h_1^s, \int_0^s \psi^u \,du )\,dt\big).
\end{align*}
In the last equality we have used the fact that $h_1^s(\mu_0)=\frac{ \partial F}{\partial s}$.
Thus, we conclude that $\textsf{i}\circ \ev_0= \id_{H^{n+1}(\pi_1^*\widetilde{\mathfrak{g}}, F^*\widetilde{\mathfrak{h}}_{[0,1]})}$ and therefore $\ev_0$ is an isomorphism.
%Indeed, choose a flat frame $e_1,\ldots,e_N$ of the bundle $\widetilde{\mathfrak{h}}$ so that 
%\begin{align*}
%& f_0^*\widetilde{\mathfrak{h}} = \OO_M\cdot e_1 \oplus\cdots\oplus \OO_M\cdot e_N \\
% & \pi_*F^*\widetilde{\mathfrak{h}}\oplus \pi_*F^*\widetilde{\mathfrak{h}}\cdot dt \\
% &=\big( \OO_{M\times [0,1]} \cdot e_1 \oplus\cdots\oplus \OO_{M\times [0,1]} \cdot e_N\big)\bigoplus \big( \OO_{M\times [0,1]} \cdot e_1 \oplus\cdots\oplus \OO_{M\times [0,1]}\cdot e_N\big)\cdot dt
 %\end{align*}
% The obstruction space is defined as the degree one cohomology of the complex
 %\[ \big( D^k(\widetilde{\mathfrak{g}},\pi_*F^*\widetilde{\mathfrak{h}}\oplus \pi_*F^*\widetilde{\mathfrak{h}}\cdot dt), d=[\mu,-]\big)\]
 %where $d\phi:= (F^*\mu_1+\nabla_{d/dt}\cdot dt)\phi-(-1)^{|\phi|'}\phi\mu_1$. 
 %Using the trivialization above, we may define extend the operator $f_0^*\mu_1$ to the bigger space $  \pi_*F^*\widetilde{\mathfrak{h}}\oplus \pi_*F^*\widetilde{\mathfrak{h}}\cdot dt $ by constant in the $t$ direction. 
 %Define a different differential $d^0$ by formula
 %\[ d^0\phi:= (f_0^*\mu_1+\nabla_{d/dt}\cdot dt)\phi-(-1)^{|\phi|'}\phi\mu_1\]
 \end{proof}

\begin{cor}
	Let $(f^0,f^{0,\sharp})$ and $(f^1,f^{1,\sharp})$ be two homotopic $L_{(n)}$ morphisms. Then $(f^0,f^{0,\sharp})$ lifts to a $L_{(n+1)}$ morphism if and only if $(f^1,f^{1,\sharp})$ does.
\end{cor}
\begin{proof}
	The morphism $(f^a,f^{a,\sharp})$ lifts to a $L_{(n+1)}$ morphism if and only if $\mathfrak{o}\big( (f^{a,\sharp}_j)_{j=1}^n\big)$ vanishes. Let $(F,F^{\sharp})$ be the homotopy between the two $L_{(n)}$ morphisms. We can easily see
	\[\mathfrak{o}\big( (f^{a,\sharp}_j)_{j=1}^n\big)=\ev_a\Big( \mathfrak{o}\big( (F^{\sharp}_j)_{j=1}^n\big)\Big).\]
	By the previous proposition, $\ev_a$ is an isomorphism therefore $\mathfrak{o}\big( (f^{a,\sharp}_j)_{j=1}^n\big)$ vanishes if and only if  $\mathfrak{o}\big( (F^{\sharp}_j)_{j=1}^n\big)$ does.
\end{proof}

The push-forward $(\mathfrak{e})_*$ and pull-back $(\mathfrak{d})^* $ maps on the obstruction space, under a $L_{(1)}$ morphism, are defined in the same manner as in the algebra case. We have the analogue to Lemma \ref{lem:hom_inv}.

\begin{lem}
	Let $\mathfrak{e}=(e,e^\sharp_1):(N,\mathfrak{h}) \to (N',\mathfrak{h}')$ and $\mathfrak{d}=(d,d^\sharp_1):(M',\mathfrak{g}') \to (M,\mathfrak{g})$ be $L_{(1)}$ homotopy equivalences. Then both maps
	\begin{align*}
	 (\mathfrak{e})_*: H^{n}(\mathfrak{g},f^*\mathfrak{h}) \to H^{n}(\mathfrak{g},(e\circ f)^*\mathfrak{h}'),\\
	 (\mathfrak{d})^*: H^{n}(\mathfrak{g},f^*\mathfrak{h}) \to H^{n}(\mathfrak{g}',(f\circ d)^*\mathfrak{h})
	 \end{align*}
	are isomorphisms.
\end{lem}
\begin{proof}
	Let $\mathfrak{E}$ be a $L_{(1)}$ homotopy between two $L_{(1)}$ morphisms $\mathfrak{e}^0$ and $\mathfrak{e}^1$. Observe that $\ev_a \circ(\mathfrak{E})_*= (\mathfrak{e}^a)_*$ for $a=0,1$. Since $\ev_a$ is an isomorphism, by Proposition \ref{prop:ev_iso}, we conclude $(\mathfrak{e}^a)_*$ is an isomorphism if and only if $(\mathfrak{E})_*$ is an isomorphism. Now let
	$\bar{\mathfrak{e}}$ be a $L_{(1)}$ homotopy inverse for $\mathfrak{e}$, then by the previous argument $(\mathfrak{e}\circ \bar{\mathfrak{e}})_*$ (and $\bar{\mathfrak{e}}\circ \mathfrak{e})_*$) is an isomorphism. Hence we conclude $(\mathfrak{e})_*$ is an isomorphism from the equality $(\mathfrak{e}\circ \bar{\mathfrak{e}})_*= (\mathfrak{e})_* \circ (\bar{\mathfrak{e}})_*$.
	
	The same argument proves the statement for $(\mathfrak{d})^*$.
\end{proof}

With these preparations, we may deduce the following result which is the analogue of Theorem~\ref{thm:A1} in the case of $L_\infty$ spaces. Its proof is essentially the same: using the previous results we prove analogues of Proposition~\ref{prop:torsor} and Lemma~\ref{obs_composition} which lead to the following theorem. Since this involves only minor modifications we omit its proof.

\begin{thm}\label{thm:A1-spaces}
An $L_\infty$ space homomorphism $\mathfrak{f}=(f, f^\sharp): (M,\mathfrak{g}) \to (N,\mathfrak{h})$ is a homotopy equivalence if and only if $(f,f^\sharp_1)$ is an $L_{(1)}$ homotopy equivalence.
\end{thm}

\begin{rmk}
Unlike in the algebraic case (see Remark~\ref{rmk:whitehead}) the Whitehead theorem of $L_\infty$ spaces (Theorem~\ref{thm:main}) does not immediately follow from the above result. The proof that quasi-isomorphisms are $L_{(1)}$ homotopies in the curved situation is considerably harder. In the remaining part of the section, we shall first prove Theorem~\ref{thm:main2} on the existence of minimal charts. Then we make use of the minimal charts to prove the desired Whitehead theorem.
\end{rmk}

\medskip
\subsection{Minimal charts}~\label{subsec:minimalchart}
%We need to introduce some terminology.
Let $\mathbb{M}=(M,\mathfrak{g})$ be a $L_\infty$ space and let $p\in M$ be a point in the zero-set of $\mu_0$. As in the Introduction we define the tangent complex of $(M,\mathfrak{g})$ at $p$ to be
\begin{equation}
T_p\mathbb{M}:= T_p M \stackrel{\nabla \mu_0|_p}{\longrightarrow} \mathfrak{g}_2|_p \stackrel{\mu_1|_p}{\longrightarrow} \mathfrak{g}_3|_p \stackrel{\mu_1|_p}{\longrightarrow} \mathfrak{g}_4|_p \stackrel{\mu_1|_p}{\longrightarrow} \ldots \stackrel{\mu_1|_p}{\longrightarrow} \mathfrak{g}_N|_p.
\end{equation}
The fact that this is indeed a complex follows from the $L_\infty$ algebra equation together with the condition $\mu_0|_p=0$. Also note that the first map is independent of the connection.

Let $\mathfrak{f}=(f, f^\sharp): (M,\mathfrak{g}) \to (N,\mathfrak{h})$ be a $L_\infty $ morphism and $p\in \mu_0^{-1}(0)$. It easily follows from the definition of morphism that $df$ and $f^\sharp_1$ induce a chain map $T_p\mathfrak{f}: T_p\mathbb{M} \to T_{f(p)}\mathbb{N}$.

\begin{defi}
	Let $(M,\mathfrak{g})$ be a $L_\infty$ space and let $p\in \mu_0^{-1}(0)$. We say $(M,\mathfrak{g})$ is \emph{minimal} at $p$ if all the maps in the complex $T_p\mathbb{M}$ are zero.
	
	A morphism $\mathfrak{f}: (M,\mathfrak{g}) \to (N,\mathfrak{h})$ is called a \emph{quasi-isomorphism} at $p$ if the chain map $T_p\mathfrak{f}$ induces an isomorphism in cohomology.
\end{defi}

We have the following easy lemma.

\begin{lem}\label{lem:qism}
	A $L_{\infty}$ morphism which is a $L_{(1)}$ homotopy equivalence is a quasi-isomorphism at any point.
\end{lem}

For any open set $W\subseteq M$ we can restrict the $L_\infty$ structure to $W$ and so obtain a new $L_\infty$ space $(W, \mathfrak{g}|_W).$ We define a chart at $p$ to be a $L_\infty$ space $(N, \mathfrak{h})$, with $n_p\in \mu_0^{-1}(0)$, together with a $L_\infty$ homotopy equivalence $\mathfrak{i}=(i,i^\sharp): (N, \mathfrak{h}) \to (W, \mathfrak{g}|_W)$ for some neighborhood $W$ of $p$ in $M$, such that $i(n_p)=p$. We say the chart is \emph{minimal} if $(N, \mathfrak{h})$ is minimal at $n_p$.

The main step in the proof of the inverse function theorem for $L_\infty$ spaces is the construction of minimal charts. We will do it in two steps.

\begin{prop}\label{prop:minimal_mu}
		Let $(M,\mathfrak{g})$ be a $L_\infty$ space and $q\in \mu_0^{-1}(0)$. There is a chart at $q$, $(N, \mathfrak{h})$ with the property $\nabla\mu_0^N|_{n_q}=0$.
\end{prop}
\begin{proof}
	In a neighborhood $U$ of $q$ with coordinates $(x_1,\ldots, x_n)$, trivialize the bundle $\mathfrak{g}_2$ and write $\mu_0=s=(s_1,\ldots, s_m): U \to \mathbb{R}^m$. If $\nabla\mu_0|_{q}\neq0$, there is $i,j$ such that $\frac{\partial s_i}{\partial x_j}(q)\neq 0$. Hence $N={s_i}^{-1}(0) \cap W$, for some small open set $W\subset U$, is a smooth submanifold.  It follows from the inverse function theorem (for smooth manifolds) that we can find coordinates on $W$, $(x_1,\ldots,x_{n-1},y)$ such that $N=\{(x_1,\ldots,x_{n-1},0)\}$. Moreover, we can decompose the bundle $\mathfrak{g}_2|_N=E_2 \oplus C_2$ such that $s(x,y)=(\nu(x,y),y)$. We define $E_k=\mathfrak{g}_k|_N$ for $k\geq 3$. We define the map $\iota: N \to M$ as $\iota(x)=(x,0)$. Additionally we denote by $i$ the inclusion $E \to \mathfrak{g}$ and by $p$ the projection $\mathfrak{g} \to E$. We claim the operations $\lambda_0:=\nu|_N$ and $\lambda_k= \mu_k|_E$, $k\geq 1$ define a $L_\infty$ space. Indeed this is a degenerate case of Theorem \ref{teorema B1} where we take $H=0$. Please note that even though Equation (\ref{eq:Hform}) does not hold on $\mathfrak{g}_2$, the theorem still holds since it is enough to have Equation (\ref{eq:Hform}) hold on $\mathfrak{g}_{\geq 3}$, since this is the only situation where it is applied.
	
	 Therefore we have an $L_\infty$ space
	\[ (N, \mathfrak{E}:=\oplus_{k\geq 2}E_k, \lambda_k).
	\]
Moreover there is a $L_\infty$ homomorphism $\iota^\sharp: \mathfrak{E} \to \iota^* \mathfrak{g}$, with $\iota^\sharp_1=i|_N$.
	We now construct a $L_{(1)}$ homotopy inverse to $(\iota, \iota^\sharp)$. For this purpose, we define the maps $\Pi: M\times [0,1]\to M$, $\Pi(x,y,t)=(x,ty)$ and $\pi^{t,\sharp}_1: \mathfrak{g} \to (\Pi^t)^*\mathfrak{g}$ by the formula
	\[ \begin{bmatrix}
	\id & -\int_t^1 \frac{\partial \nu}{\partial y}(x,sy)\,ds \\
		0 & t\cdot\id 
	\end{bmatrix}: \mathfrak{g}_2 \to (\Pi^t)^*\mathfrak{g}_2 ,
	\]
	and $\pi^{t,\sharp}_1=\id$ on  $\mathfrak{g}_{k\geq 3}$. We claim the pair $(P, P^\sharp_1)$, where $P(x,y)=x$ and $P^\sharp_1=p \pi^{0,\sharp}_1$, is a $L_{(1)}$ morphism from $(M,\mathfrak{g})$ to $(N,\mathfrak{E})$, and  moreover it is a $L_{(1)}$ homotopy inverse to $(\iota, \iota^\sharp)$. We first show that $\pi^{t,\sharp}_1$ is a $L_{(1)}$ homomorphism. An easy computation gives
	\[\pi^{t,\sharp}_1(\mu_0)= (\nu(x,ty), ty)= (\Pi^t)^*(\mu_0).
	\]
	In the decomposition $\mathfrak{g}_2|_N=E_2 \oplus C_2$, we write $\mu_1= (\varphi, \alpha)$ and compute
	\[\big( \pi^{t,\sharp}_1\mu_1-\mu_1^t\pi^{t,\sharp}_1 \big)|_{\mathfrak{g}_2} = (\varphi - \varphi^t , \alpha - t \alpha^t + \varphi^t \cdot \int_t^1 \frac{\partial \nu}{\partial y}(sy)\,ds).
	\]
	We claim this is $\delta$-exact. We define	
	\[P_2=\int_t^1 \mu_2^s(K \otimes \id) -  \mu_2^s( \id\otimes K)\,ds,
	\]
	where $K: \mathfrak{g}_2 \to T_M$ is the map defined as $K(e,c)=c\frac{\partial}{\partial y}$, and compute
	\begin{align}
	\delta(P_2)(e,c) & = \int_t^1 \mu^s_2(y\frac{\partial}{\partial y},  (e,c) )- \mu^s_2( \mu_0 ,  c\frac{\partial}{\partial y})\, ds \nonumber \\
	& = \int_t^1 \nabla_{\frac{d}{d s}} \mu^s_1((e,c)) - c \frac{\partial\mu_1}{\partial y}(x,sy)(\mu_0)\, ds\\
	& = \big( \varphi - \varphi^t , \ \alpha - \alpha^t - \int_t^1 \frac{\partial\varphi}{\partial y}(sy)\cdot\nu(x,y)\,ds - \int_t^1 \frac{\partial\alpha}{\partial y}(sy)y\,ds \big) (e, c). \nonumber
	\end{align}
	Therefore $\delta(P_2)= \big( \pi^{t,\sharp}_1\mu_1-\mu_1^t\pi^{t,\sharp}_1 \big)|_{\mathfrak{g}_2}$ is equivalent to the following identity
	\[ \int_t^1 \varphi(x,ty) \cdot \frac{\partial \nu}{\partial y}(sy)\,ds + \int_t^1 \frac{\partial\varphi}{\partial y}(sy)\cdot\nu(x,y)\, ds = t \alpha^t - \alpha,
	\]
	which in turns follows from the fact that the left-hand side equals
	\begin{align}
	\int_t^1 \frac{\partial(\varphi\cdot\nu)}{\partial y}(x,sy)\, ds.
	\end{align}
	The $L_\infty$ relation $\mu_1(\mu_0)=0$ implies $\varphi\cdot\nu = -\alpha\cdot y $, hence
		\begin{align}
	\int_t^1 \frac{\partial(\varphi\cdot\nu)}{\partial y}(x,sy)\, ds= -	\int_t^1 \frac{\partial\alpha}{\partial y}(x,sy)sy + \alpha(x,sy)\, ds= t\alpha(x,ty) - \alpha(x,y).\nonumber
	\end{align}
	Here the last equality is given by integration by parts. Similarly one shows that $\big( \pi^{t,\sharp}_1\mu_1-\mu_1^t\pi^{t,\sharp}_1 \big)|_{\mathfrak{g}_{\geq 3}} = \delta(\int_t^1 \mu_2^s(K \otimes \id))$. Hence we conclude that $\pi^{t,\sharp}_1$ is a $L_{(1)}$ homomorphism.
	
	We observe that $P^\sharp \circ \iota^\sharp=\id$, $\pi^{0,\sharp}_1= \iota^\sharp \circ P^\sharp$ and $\pi^{0,\sharp}_1= \id$. Therefore, in order to conclude that $(P, P\sharp)$ is a $L_{(1)}$ homotopy inverse to $(\iota, \iota^\sharp)$ it is enough, by (\ref{eq:homotopy}), to define
	\[  h_1^t= \left\{ {\begin{array}{ll}
		K, & \mathfrak{g}_2 \\
		0, & \mathfrak{g}_{\geq 3}\\
		\end{array} } \right.
	\]
	and check the following identities
	\[h^t_1(\mu_0)= \frac{\partial \Pi}{\partial t},  \ \ \frac{\partial \pi^{t,\sharp}_1}{\partial t} = \mu_1^t h_1^t + h^t_1 \mu .\]
	
	We have thus proved that the $L_\infty$ morphism $(\iota, \iota^\sharp)$ is a $L_{(1)}$ homotopy equivalence. It now follows from Theorem~\ref{thm:A1-spaces} that $(\iota, \iota^\sharp)$ is a $L_\infty$ homotopy equivalence and therefore $(N, \mathfrak{E})$ is a chart at $q$. By construction, the rank of $\nabla\mu_0|_q$ in $n$ is strictly smaller than in the original space $M$, hence by applying the previous construction finitely many times we can find a chart at $q$ such that $\nabla\mu_0|_q$ has rank zero as claimed in the statement.
\end{proof}	

We now state and prove Theorem \ref{thm:main2} in the Introduction.

\begin{thm}\label{thm:minimal}
		Let $(M,\mathfrak{g})$ be a $L_\infty$ space and $q\in \mu_0^{-1}(0)$. There is a minimal chart at $q$.
\end{thm}
	\begin{proof}
	Proposition \ref{prop:minimal_mu} implies that we can assume $\nabla\mu_0|_q=0$. 
	We pick a sub-bundle $E_3 \subseteq \mathfrak{g}_3$ such that $E_3|_q \oplus \imag(\mu_1|_q)=\mathfrak{g}_3|_q$. Then  $\widetilde{\mu_1}:  \mathfrak{g}_2 \to \mathfrak{g}_3/E_3$ is surjective at $q$, since $\imag \mu_1|_q=\imag \widetilde{\mu_1}|_q $. Hence $\widetilde{\mu_1}$ is surjective on some neighborhood of $q$, which we denote by $U$. This implies that $A_2:=\ker \widetilde{\mu_1}|_{U}$ is a sub-bundle of $\mathfrak{g}_2$. We pick a complement $A_2\oplus B_2= \mathfrak{g}_2$ and define $C_3=\mu_1(B_2)\subseteq\mathfrak{g}_3|_{U}$. By construction, $\mu_1|_{B_2}$ is injective and therefore $C_3$ is a bundle. Moreover $\mathfrak{g}|_3 = E_3 \oplus C_3$.
	
	Next we pick $E_4\subset \mathfrak{g}_4$ such that $E_4|_q \oplus \imag(\mu_1|_q)=\mathfrak{g}_4|_q$ and define $A_3:=\ker \widetilde{\mu_1}: E_3 \to \mathfrak{g}_4/E_4$. We pick a complement $A_3 \oplus B_3=E_3$ and define $C_4:=\mu_1(B_3)$. We repeat this argument, for all $k$ and obtain a decomposition $\mathfrak{g}_k=A_k\oplus B_k \oplus C_k$ in a neighborhood of $q$, here $C_2=0$. In this decomposition, we the map $\mu_1$ takes the form
	\[ \mu_1=\begin{bmatrix}
	\varphi & 0 & \alpha \\
	\psi & 0 & \beta \\
	0 & \epsilon & \gamma 
	\end{bmatrix}
	\]
	Note that $\epsilon$ is an isomorphism, so we can define the degree $-1$ map
	\[ H=\begin{bmatrix}
	0 & 0 & 0 \\
	0 & 0 & -\epsilon^{-1} \\
	0 & 0 & 0 
	\end{bmatrix}: \mathfrak{g} \to \mathfrak{g}.
	\]
	On $\mathfrak{g}_2$, $H$ is defined to be zero.
	We denote by $i$ the inclusion $A \to \mathfrak{g}$ and by $p$ the projection onto $A$. We have the following identities on $\mathfrak{g}$, which are easy to check,
	\begin{align}
	H\circ i =0,  \ &  \ p\circ H=0\\
	H \mu_1 + \mu_1 H & = i\circ p - \id_{\mathfrak{g}} - H \mu_1^2 H.\label{eq:Hform}
	\end{align}
	The first $L_\infty$ equation $\mu_1(\mu_0)=0$ implies that $\mu_0=(\nu, 0)$, since $\epsilon$ is an isomorphism. Hence $i(\nu)|_N= \iota^*\mu_0$ and we have all the data and conditions in Theorem \ref{teorema B1}. Therefore Theorem \ref{teorema B1} constructs an $L_\infty$ space
	\[ (M, \mathfrak{A}:=\oplus_{k\geq 2}A_k, \lambda_k),
	\]
	with $\lambda_0=\nu$ and $\lambda_1= p \mu_1 i$. Moreover there is a $L_\infty$ morphism $(\iota,\iota^\sharp): (U,\mathfrak{A}) \to (U, \mathfrak{g}|_U$, with $\iota=\id$ and $\iota^\sharp_1=i$. Also observe that, by definition of $A_k$, $\mu_1|_q(A_k)=0$ and thus $(U, \mathfrak{A})$ is minimal at $q$.
	
	The last step is to construct a $L_{(1)}$ homotopy inverse to $(\iota, \iota^\sharp)$ and then appeal to Theorem~\ref{thm:A1-spaces} to conclude that $(\iota, \iota^\sharp)$ is a $L_\infty$ homotopy equivalence. For this purpose, we define the maps  $\pi^{t,\sharp}_1: \mathfrak{g} \to \mathfrak{g}$ by the formula
		\[ \begin{bmatrix}
	\id & 0 & 0 \\
	0 & t\cdot\id &  0\\
	0 & 0 & t\cdot\id 
	\end{bmatrix}: \mathfrak{g} \to \mathfrak{g},
	\]
	We first show that $\pi^{t,\sharp}_1$ is a $L_{(1)}$ homomorphism. Since $\mu_0=(\nu, 0)$ we have $\pi^t_1(\mu_0)= \mu_0$.
	Next we define the map
	\[P_2= (1-t)\Big(H\mu_2(p \otimes p) + p\mu_2(H\otimes \id)- p\mu_2(\id \otimes H) \Big).
	\]
	A simple computation using (\ref{eq:Hform}) gives 
	\[\mu_1\pi^{t,\sharp}_1- \pi^{t,\sharp}_1\mu_1 = \delta(P_2).
	\]
	In particular, we have shown that $P^\sharp_1:=\Pi\pi^{0,\sharp}$ (where $\Pi: \mathfrak{g} \to A $ is the projection) is a $L_{(1)}$ morphism from $\mathfrak{g}$ to $\mathfrak{A}$. 
	
	It is obvious that $P_1^\sharp \circ \iota^\sharp=\id_\mathfrak{A}$. Finally, we need to show that $\iota^\sharp \circ P^\sharp_1$ is $L_{(1)}$ homotopic to the identity. First note $\pi^{1,\sharp}=\id_\mathfrak{g}$ and $\pi^{0,\sharp}=\iota^\sharp \circ P^\sharp_1$. We define $h_1^t=-H$ and easily check
	\[ h_1^t(\mu_0)=0, \ \ \frac{\partial \pi^{t,\sharp}}{\partial t}= h_1^t\mu_1 + \mu_1 h_1^t- \delta(Q_2),
	\] 
	where $Q_2=H\mu_2(H\otimes \id)+ H\mu_2(\id \otimes H)$. This completes the proof that  $(\iota, \iota^\sharp)$ is a $L_{(1)}$ and therefore a  $L_{\infty}$ homotopy equivalence.  
\end{proof}

We are now ready to prove Theorem~\ref{thm:main} in the Introduction.

\begin{thm}
	Let $(M,\mathfrak{g})$ and $(N,\mathfrak{h})$ be  $L_\infty$ spaces and $\mathfrak{f}=(f, f^\sharp): (M,\mathfrak{g}) \to (N,\mathfrak{h})$ be a $L_\infty $ morphism. Assume that $\mathfrak{f}$ is a quasi-isomorphism at $q\in \mu_0^{-1}(0)$. Then there are neighborhoods $U$ of $q$ and $V$ of $f(q)$ such that  such that $f(U)\subseteq V$ and 
	\[\mathfrak{f}|_{U}: (U,\mathfrak{g}|_{U}) \to (V,\mathfrak{h}|_V)
	\]
	is a $L_\infty$ homotopy equivalence.
\end{thm}
\begin{proof}
	Theorem~\ref{thm:minimal} provides minimal charts at $p$ and $f(p)$ and hence we have the following diagram
	\[
	\begin{tikzcd}[every arrow/.append style={shift left}]
	(L, \mathfrak{a})\arrow[d,"\mathfrak{J}"] \arrow{r}{\mathfrak{i}} &(U, \mathfrak{g}|_{U})  \arrow{l}{{\mathfrak{p}}} \arrow{d}{\mathfrak{f}_U} \\
	(\widetilde{L}, \widetilde{\mathfrak{a}})\arrow[u,"\mathfrak{J}^{-1}", dotted] \arrow{r}{\widetilde{\mathfrak{i}}} &(V, \mathfrak{h}|_{V})  \arrow{l}{\widetilde{\mathfrak{p}}} 
	\end{tikzcd}
	\]
	Here the pairs $\mathfrak{i}, \mathfrak{p}$ and $\widetilde{\mathfrak{i}}, \widetilde{\mathfrak{p}}$ are homotopy inverses and $\mathfrak{J}:=\widetilde{\mathfrak{p}}\circ \mathfrak{f}_U\circ \mathfrak{i}$. It follows from Lemma \ref{lem:qism} that $\mathfrak{J}$ is a quasi-isomorphism at $n_q=:p(q)$, but since $(L, \mathfrak{a})$ and 	$(\widetilde{L}, \widetilde{\mathfrak{a}})$ are minimal at $n_q$ and $J(n_q)$, we conclude that $J$ is a local diffeomorphism and $J_1^\sharp|_{n_q}$ is an isomorphism. After restricting to small neighborhoods $U'$ and $V'$ of $n_q$ and $J(n_q)$ we have that $J$ is a diffeomorphism and $J_1^\sharp$ is an isomorphism of bundles. Therefore we can solve Equation (\ref{eq:comp}) inductively on $n$ to find a strict $L_\infty$ inverse to $J^\sharp$, which we denote by $\mathfrak{J}^{-1}$.
	
	We make the neighborhoods $U$ and $V$ smaller, if necessary, to ensure the restrictions of $\mathfrak{i}, \mathfrak{p}$ (and $\widetilde{\mathfrak{i}}, \widetilde{\mathfrak{p}}$) are homotopy inverses on $(U', \mathfrak{a}|_{U'})$ (and $(V', \widetilde{\mathfrak{a}}|_{V'})$). We define $\mathfrak{K}:=\mathfrak{i}\circ\mathfrak{J}^{-1}\circ\widetilde{\mathfrak{p}}$. By construction $\mathfrak{J}^{-1} \circ\widetilde{\mathfrak{p}}\circ\mathfrak{f}_{U}\cong \mathfrak{p}$, therefore 
	$$\mathfrak{K}\circ \mathfrak{f}_{U} \cong \mathfrak{i}\circ\mathfrak{J}^{-1}\circ\widetilde{\mathfrak{p}}\circ \mathfrak{f}_{U}\cong \mathfrak{i}\circ\mathfrak{p}\cong \id.$$
	Similarly, $ \mathfrak{f}_{U} \circ\mathfrak{i}\circ\mathfrak{J}^{-1} \cong \widetilde{\mathfrak{i}}$, hence $ \mathfrak{f}_{U} \circ\mathfrak{K}\cong  \widetilde{\mathfrak{i}}\circ \widetilde{\mathfrak{p}}\cong \id$. Thus $\mathfrak{K}$ is a homotopy inverse to $\mathfrak{f}_{U}$.

\end{proof}


\begin{thebibliography}{99}

\bibitem{Amo} Amorim, L., \emph{Tensor product of filtered $A_\infty$-algebras}. Journal of Pure and Applied Algebra, 220, no. 12, 3984--4016. (2016)

\bibitem{Beh1} Behrend, K., {\em  Differential Graded Schemes I: Perfect Resolving Algebras.} arXiv:math/0212225, 2002.

\bibitem{Beh2} Behrend, K., {\em Differential Graded Schemes II: The 2-category of Differential Graded Schemes.} arXiv:math/0212226, 2002.

\bibitem{BLX} Behrend, K.; Liao H.; Xu, P., \emph{Derived Differential Geometry.} arXiv:2006.01376, 2020.

\bibitem{BBJ} Brav, C.; Bussi, V.; Joyce, D., \emph{A Darboux theorem for derived schemes with shifted symplectic structure}, Journal of the AMS 32 (2019), 399--443.

\bibitem{CFK1} Ciocan-Fontanine, I.; Kapranov, M., {\em Derived Quot schemes.} Ann. Sci. Ecole Norm. Sup. (4), 34(3): 403--440, 2001.

\bibitem{CFK2} Ciocan-Fontanine, I.;  Kapranov, M., {\em Virtual fundamental classes via dg-manifolds.} Geom. Topol., 13(3):1779--1804, 2009.

\bibitem{Cos} Costello, K., {\em A geometric construction of the Witten genus, II}.  arXiv:math/1112.0816, 2011.

\bibitem{DHR} Dolgushev, V.; Hoffnung, A.; Rogers, C., {\em What do homotopy algebras form?} Adv. Math. 274 (2015), 562--605.

\bibitem{FM} Fiorenza, D.; Manetti, M., {\em L-infinity structures on mapping cones.}  Algebra and Number Theory, Vol. 1, No. 3 (2007), 301--330.

\bibitem{F} Fukaya, K., {\em Deformation theory, homological algebra and mirror symmetry.} Geometry and physics of branes (Como, 2001), 121--209, Ser. High Energy Phys. Cosmol. Gravit., IOP, Bristol, 2003. 

\bibitem{FOOO} Fukaya, K.; Oh, Y.; Ohta, H.; Ono, K., {\em Lagrangian intersection Floer theory: anomaly and obstruction}. Parts I and II, vol. 46, AMS/IP Studies in Advanced Mathematics, American Mathematical Society, Providence, RI, 2009.

\bibitem{LO} Le, H.-V.; Oh, Y.-G., {\em Deformations of coisotropic submanifolds in locally conformal symplectic manifolds. }
Asian J. Math. 20 (2016), no. 3, 553--596.

\bibitem{Joy} Joyce, D., \emph{A new definition of Kuranishi space}. arXiv:1409.6908, 2014.

\bibitem{LodVal} Loday, J.-L.; Vallette, B., \emph{ Algebraic operads.} Grundlehren der Mathematischen Wissenschaften [Fundamental Principles of Mathematical Sciences], 346. Springer, Heidelberg, 2012. xxiv+634 pp. 

\bibitem{markl} Markl, M., \emph{Transferring $A_\infty$ (strongly homotopy associative) structures}. Rend. Circ. Mat. Palermo (2) Suppl. (79), 139--151. (2006)

\bibitem{marklB} Markl, M., \emph{Deformation theory of algebras and their diagrams.} CBMS Regional Conference Series in Mathematics, 116. Published for the Conference Board of the Mathematical Sciences, Washington, DC; by the American Mathematical Society, Providence, RI, 2012. x+129 pp. ISBN: 978-0-8218-8979-4

\bibitem{OP} Oh, Y.; Park, J., \emph{Deformations of coisotropic submanifolds and strong homotopy Lie algebroids.} Invent. Math. 161 (2005), no. 2, 287--360.

\bibitem{Pro} Proute, A., \emph{ Alg\`ebres diff\' erentielles fortement homotopiquement associatives}, Th\`ese d'Etat, Universit\' e Paris VII, 1984.

\bibitem{PS} Pym, B.; Safronov, P., {\em Shifted Symplectic Lie Algebroids.} Int. Mat. Res. Not., rny215, https://doi.org/10.1093/imrn/rny215.

\bibitem{TV} To\"en, B.;  Vezzosi, G., {\em Homotopical algebraic geometry. II. Geometric stacks and applications.} Mem. Amer. Math. Soc.,193(902):x+224, 2008.

\bibitem{Tu} Tu, J., {\em Homotopy L-infinity spaces.} arXiv:1411.5115, 2014.

\bibitem{V}  Vallette, B., {\em Homotopy theory of homotopy algebras}. Ann. Inst. Fourier (Grenoble) 70 (2020), no. 2, 683--738.

\end{thebibliography}
\end{document}